\documentclass[a4paper,11pt]{article}
\usepackage{amssymb}
\usepackage{graphicx}
\usepackage{tikz}
\usetikzlibrary{intersections}
\usetikzlibrary{svg.path}
\usetikzlibrary{decorations.pathmorphing}

\usepackage{stackrel}
\usepackage{pgfplots}
\pgfplotsset{compat=1.11}
\usepackage{mathrsfs}
\usepackage{amsmath}
\usepackage{amsthm}
\usepackage{enumerate}
\usepackage{color}
\usepackage{verbatim}
\usepackage{todonotes}

\usepackage[margin=2cm, includefoot, footskip=30pt]
{geometry}

\usepackage{array}
\newcolumntype{e}{>{\displaystyle}r @{\,} >{\displaystyle}c @{\,} >{\displaystyle}l}

\usepackage{amssymb,amsopn,amscd}

\usepackage[colorlinks]{hyperref}
\hypersetup{final}
\usepackage{xcolor}
\hypersetup{
    colorlinks,
    linkcolor={blue!40!black},
    citecolor={blue!40!black},
    urlcolor={blue!80!black}
}
\usepackage{nameref}

\usepackage[figure,table]{hypcap} 

\theoremstyle{plain}
\newtheorem{Theorem}{Theorem}[section]
\newtheorem{Corollary}[Theorem]{Corollary}
\newtheorem{Lemma}[Theorem]{Lemma}
\newtheorem{Proposition}[Theorem]{Proposition}

\theoremstyle{definition}
\newtheorem{Example}[Theorem]{Example}
\newtheorem{Definition}[Theorem]{Definition}
\newtheorem{Remark}[Theorem]{Remark}

\newcommand{\T}{\mathbb{T}}

\newcommand{\N}{\mathbb{N}}

\newcommand{\M}{\mathcal{M}}

\newcommand{\m}{\mathbf{m}}

\newcommand{\bz}{\boldsymbol{\zeta}}

\newcommand{\B}{\mathbf{B}}
\newcommand{\Pbrw}{\mathbb{P}}
\newcommand{\Prw}{\mathbf{P}}
\newcommand{\Ebrw}{\mathbb{E}}
\newcommand{\Erw}{\mathbf{E}}
\newcommand{\cM}{\mathcal{M}}

\numberwithin{equation}{section}

  \newcounter{constant}

\newcommand*\mycirc[1]{%
\begin{tikzpicture}[baseline=(C.base)]
\node[draw,circle,inner sep=1pt](C) {#1};
\end{tikzpicture}}

\begin{document}

\title{Martin boundaries and asymptotic behavior \\ of branching random walks}

\author{Daniela Bertacchi\footnote{daniela.bertacchi@unimib.it; Universit\`a Milano Bicocca, Milan}
\and
Elisabetta Candellero\footnote{elisabetta.candellero@uniroma3.it;  University of Roma Tre, Rome} 
\and 
Fabio Zucca\footnote{fabio.zucca@polimi.it;  Politecnico di Milano, Milan}
}

\maketitle

\begin{abstract}
Let $G$ be an infinite, locally finite 
graph.
We investigate the relation between supercritical, transient branching random walk
and the Martin boundary of its underlying random walk.
We show results regarding the typical asymptotic directions taken by the 
particles, and as a consequence we find a new connection between $t$-Martin boundaries and standard Martin boundaries.
Moreover, given a subgraph $U$ we study two aspects of branching random walks on $U$: when the trajectories visit $U$ infinitely often (survival) and when they stay inside $U$ forever (persistence).
We show that there are cases, when $U$ is not connected, where the branching random walk alamost surely does not survive in $U$, but the random walk on $G$ converges to the boundary of $U$ with positive probability.
In contrast, the branching random walk can survive with positive probability 
in $U$ even though the random walk eventually exits $U$ almost surely.
We provide several examples and counterexamples. 
%
%
%
%
%
%
\end{abstract}




\section{Introduction}

The remarkable connection between the asymptotic behavior of a random walk (RW) on a graph and its Martin boundary is now a classical fact in the probabilistic literature (see e.g.~\cite{NeySpitzer},  \cite{sawyer_harmonic_functions}, \cite[Chapter IV]{woess2000}, \cite{Woess-denumerable} and references therein).
Since a branching random walk (BRW) is defined using a RW transition kernel, it is natural to ask what connections there are between the asymptotic behavior of a BRW and the Martin boundary of the associated RW.
Great interest has been given to the study of Martin boundaries in several settings, for example, in the case of branching processes (without transition kernel) we recall the results in \cite{Lootgieter} and \cite{Overbeck}, that study the convergence of \emph{space-time} Martin kernels.
For more recent results we refer the reader to \cite{Abraham-Delmas} and references therein.

Lately there has been increasing interest in the study of the long-time behavior of BRWs on graphs.
Particular attention has been given to the investigation of the limiting set of the \emph{trace}, i.e.\ the random subgraph visited by particles of the BRW, see for example \cite{candellero_roberts_BRW}, \cite{GilchMueller-EndsBRW}, and \cite{Hutchcroft-NonIntersectionBRW}.
In particular, \cite{Hutchcroft-NonIntersectionBRW} shows that the number of ends of the trace of a transient BRW on a Cayley graph of a non-amenable group is infinite.

The first works relating BRWs and Martin boundaries are \cite{CandelleroHutchcroft} and \cite{KW2022}, which point out a connection between the Martin boundary of a RW, which we denote by $\M(P,1)$, and BRWs with the same transition kernel.

In this note we proceed the investigation initiated in \cite{CandelleroHutchcroft} 
by studying the asymptotic behavior of a supercritical and transient BRW.
More precisely, on an infinite, locally finite 
graph consider a BRW with transition kernel $P$ and offspring distribution $\nu$: each individual lives for one unit of time and then produces a random number of offspring according to $\nu$; the offspring are independently dispersed according to $P$.
We find 
conditions that guarantee that there are trajectories of BRW that converge to a subset of the Martin boundary for RW governed by $P$
(see Theorem \ref{prop:equivalent}).
We consider the behavior of the BRW on subgraphs $U$.
If $U$  is ``regular enough'' (in a sense that will be made precise later), we identify a \emph{phase transition} in the mean $\m:=\sum_{k\geq 0} k\nu(k)$, in the following sense.
We find a threshold $\m_1(U)>0$ such that, whenever $\m< \m_1(U)$, then almost surely all trajectories exit $U$; whereas when $\m >\m_1(U) $, then with positive probability there are trajectories that never exit $U$ (see Corollary \ref{cor:m1+}).

In order to state our results we need to introduce some notation.

\subsection{Notation.}
The probability measure and the expectation of a BRW depend on the transition kernel, the offspring distribution and the initial state of the process. To avoid confusion, sometimes we write these dependencies explicitly.
More precisely, given a transition kernel ${P}$ defined on $G$ we denote the probability measure and expectation of the corresponding BRW as $\Pbrw^{{P}}$ and $\Ebrw^{{P}}$ respectively, unless specified otherwise.
From now on, with a slight abuse of notation, for all graphs $X$, by $x\in X$ we mean that $x$ is a vertex of $X$.
If the initial state consists of one particle at $x \in G$, then we write $\Pbrw_x$ and $\Ebrw_x$.
Furthermore, for every $n\geq 1$,  $\B_n\in {\N}^{G}$ corresponds to the infinite-dimensional vector (with only finitely many non-zero entries), so that each entry in position $x\in G$ gives the number of particles of BRW that are alive at $x$ at time $n$.
%
In this paper we frequently deal with RWs as well. Although it is possible to assume that all the processes are defined on a unique probability space, it is preferable to use separate notations for the probability measures and expectations  related to the BRW (namely $\Pbrw$ and $\Ebrw$) and the probability measures and expectations related to the RW, $\Prw$ and $\Erw$.

\subsection{Transition matrices.}
Throughout this work we will use several transition kernels, which we define here.
The ``main'' one is defined on $G$ and denoted by $P:=\big (p(x,y)\big )_{x,y}$.
We assume that $P$ is \emph{irreducible} and \emph{nearest-neighbor}, even though our results can be easily modified to cover the cases of finite-range transition kernels. More precisely, we have that $p(x,y)>0$ if and only if $x$ and $y$ are neighbors.
Let $P^n=\big (p^{(n)}(x,y)\big )_{x,y}$ denote the $n$-th  convolution power of $P$, then the spectral radius of $P$ is
\begin{equation}\label{eq:rhoG}
\rho_G:=\limsup_n \left ( p^{(n)}(x,x)\right )^{1/n} .
\end{equation}
It is well known that if $\m \le 1$ then the BRW dies out almost surely, while if $\m > 1/\rho_G$ then the BRW visits infinitely many times every finite set with positive probability. The most interesting case for our purposes is when $\rho_G<1$: indeed, if $1 < \m \le 1/\rho_G$ then the BRW almost surely vacates any finite subset,
but with positive probability there are particles alive at any time; in this case we say that the BRW is \emph{transient}.
From now on, if not otherwise stated, we  assume that $\rho_G<1$ and that $1<\m \leq  1/\rho_G$.

For every connected subgraph $U\subseteq G$ define the substochastic matrix $P_U:=\big (p_U(x,y)\big )_{x,y}$ so that $p_U(x,y):=p(x,y)$ whenever $x,y \in U$ and $p_U(x,y):=0$ otherwise.
The spectral radius of $P_U$ is
\begin{equation}\label{eq:defRhoU}
	\rho_U:=\limsup_n \left ( p_U^{(n)}(x,x)\right )^{1/n}.
\end{equation}
From now on, we tacitly assume the following property for subgraphs $U$:
\begin{equation}
	\forall \ x,y\in U 
	\text{ we have }\{x,y\}\in E(U) \Longleftrightarrow \{x,y\}\in E(G),
\end{equation}
i.e., whenever two vertices in $U$ are neighbors in $G$, we also assume that they are neighbors in $U$.
Moreover, we assume that $U$ is not a singleton $\{x_0\}$, with $p(x_0,x_0)=0$.
Note that $\delta_x:=\sum_{w\in U}p_U(x,w)>0 $ for all $x\in U$; we define a new matrix $Q_U:=\big (q_U(x,y)\big )_{x,y\in U}$ as follows. 
For each fixed $x\in U$, $q_U(x,y):= p_U(x,y)/\delta_x$.
By definition we have $\sum_{w\in U}q_U(x,w)=1$ for all $x\in U$.
Each entry $q_U(x,y)$ represents the probability that a RW with transition kernel $P$ moves from $x$ to $y$, \emph{conditional} on the event that it cannot exit $U$.
In fact, whenever the RW is transitioning between internal vertices of $U$ it moves according to $P_U$ (i.e.\ $P$), whereas when it is moving from a vertex on the boundary of $U$ (thus possibly exiting towards other regions of $G$), then it is forced to stay inside.
The spectral radius of the walk governed by $Q_U$ is
\begin{equation}\label{eq:phi_U}
\phi_U:=\limsup_n \left ( q_U^{(n)}(x,x)\right )^{1/n}  .
\end{equation}
By irreducibility, none of $\rho_G, \rho_U,\phi_U$ depend on the starting vertex $x$. 
Moreover, by the supermultiplicative property 
and Fekete's Lemma, they are all strictly positive.

\subsection{Persistence and survival in subgraphs.}

In the following we use a concept which will be defined in Section \ref{sect:constr-propBZ}, but can be informally understood as follows.
Consider the original BRW on $G$ started at a fixed vertex $x\in U$, but at each step \emph{kill} all particles exiting $U$.
The resulting process is still a BRW; we  refer to it as the ``\emph{BRW induced by $U$}'' and denote it by $\{\mathbb{U}_n\}_n$.
Its offspring distribution
depends on the structure of $U$ and in general it is different at different vertices of $U$.
We say that a BRW defined on $G$ \emph{persists globally} in $U$ if $\{\mathbb{U}_n\}_n$ has alive particles at all times, with positive
probability.
We speak of \emph{local persistence} at vertex $x\in U$ if  $x$ is visited infinitely often by $\{\mathbb{U}_n\}_n$ with positive probability (see Definition~\ref{def:global-local}).

We say that the BRW starting from $x \in G$  \emph{survives} in a subset $U \subseteq G$, if $\{\B_n\}_n$ has positive probability of having an infinite number of particles visiting $U$  (see Definition~\ref{def:survival-U}). 
Clearly, persistence implies survival.
For all $U\subseteq G$ and all $x\in U$ we denote by $\mathbf{q}(x,U)$ the probability that $\{\B_n\}_n$ started at $x$ visits $U$ a finite number of times and we call it extinction probability.  We say that there is \emph{global survival} starting from $x$ if $\mathbf{q}(x,G)<1$;  we say that there is \emph{local survival} starting from $x$ if $\mathbf{q}(x,\{y\})<1$ for all $y \in G$.

It will be important to know when the BRW induced by a fixed subset $U$ is 
an 
$\mathcal{F}$-BRW (see \cite{Zucca}, \cite{BZ_genfunapproach} and \cite{bertacchi-zucca-infinite-type}).
Roughly speaking, in this case (see Definition \ref{def:F-BRW}), the global behavior of the BRW induced by $U$ is the same as the global behavior of a multi-type Galton-Watson process with only \emph{finitely} many types (see \cite[Definition 2.2]{Zucca} and the subsequent discussion). Sometimes, for sake of simplicity, we say that an $\mathcal{F}$-BRW has a finite number of types. Note that $\{\B_n\}_n$ has only one type, whereas $\{\mathbb{U}_n\}_n$ has only
one type if and only if the (induced) offspring distribution of $\{\mathbb{U}_n\}_n$
does not depend on the vertex
(in this case the amount of ``lost'' children which would be placed 
outside $U$ by $\{\B_n\}_n$ at any vertex has the same law everywhere).
Note that transitive and quasi-transitive BRWs are particular cases of $\mathcal{F}$-BRWs.
Let us point out that even if we choose $U$ to be a subset of an infinite, locally finite 
and transitive 
graph, and the BRW induced by $U$ 
is an $\mathcal F$-BRW, that does not imply that
$U$ (as a graph), or the induced BRW,
are quasi-transitive.
However, if $P$ is the simple RW on $G$,
then any BRW induced by constant-degree subgraphs
is an $\mathcal F$-BRW.

Besides survival and persistence in a subgraph $U$, we are interested in the limiting behavior of the BRW. This leads us to wonder on 
the one hand
whether there is convergence to the Martin boundary of $U$, and on the other hand whether, given a subset $A$ of the Martin boundary
of $G$, there is convergence of trajectories to $A$. 

We need some notation.
We denote by $\M(P,1)$ the Martin boundary for RW on $G$ governed by the transition kernel $P$.
For any Borel set $A\subseteq \M(P,1)$,
whenever we refer to a \emph{neighborhood} $U$ of $A$, we mean a neighborhood with respect to the usual topology on $\M(P,1)$, or any equivalent formulation, see e.g.\ \cite{picardello_woess_trees} or \cite{picardello_woess}.

Let $d_G(\cdot, \cdot)$ denote the distance in $G$ induced by the graph metric
and fix a reference vertex $o$.
For any set $U\subseteq G$, let
\[
\widehat{U}:=U\cup \{\xi \in \M(P,1) \ : \ \exists \text{ sequence }\{y_n\}_n\in U \text{ s.t.\ }d_G(o,y_n)\to \infty, \text{ and } y_n\to \xi\},
\]
where $y_n\to \xi$ denotes convergence in the topology of the Martin boundary (see \cite[Chapter 7]{Woess-denumerable} for details).

\subsection{Main results.}

Our first result, Theorem~\ref{prop:equivalent}, describes a relation between global persistence of supercritical BRW in a subset $U$ and the spectral radii $\rho_U$ and $\phi_U$.
Theorem~\ref{prop:equivalent} also investigates the relationship between the 
underlying RW, the Martin boundary  $\M(P,1)$, persistence and survival.
More precisely, we first take a Borel set
$A \subseteq \M(P,1)$, such that the RW has a positive probability of convergence to $A$, and show that the BRW persists in any connected neighborhood of $A$. 
Note that, to ensure persistence with positive probability, we need a neighborhood of $A$, not just any $U$ such that $\partial U=A$.
Nevertheless, it is true that on trees, if $U$ is connected and the RW has a positive probability of convergence to 
its topological boundary $\partial U \subseteq \M(P,1)$, then there is positive probability of survival in $U$.

The probability of convergence to a Borel set $A \subseteq \M(P,1)$  for a RW starting from $o$ is denoted by $\gamma_o(A)$, where $\gamma_o$ is called the \emph{harmonic measure} associated to the RW started at $o$.


\begin{Theorem}\label{prop:equivalent}

\begin{itemize} \leavevmode
\item[(i)] 
For a fixed connected subgraph $U$,
let $x\in U$ be a fixed vertex.
If for all $\m\in (1, \rho_G^{-1}]$ we have positive probability of global persistence in $U$, then $\rho_U=\phi_U$. 
Furthermore, if the BRW induced by $U$ is 
an $\mathcal{F}$-BRW, then also vice versa holds.
\item[(ii)] Let $A\subseteq \M(P,1)$ be a Borel set and $U$ a fixed connected neighborhood of $A$. If  $\gamma_o(A)>0 $, then for all $\m>1$, 
there is persistence in $U$ with positive probability starting from every $x \in U$.
\item[(iii)] Suppose that $G$ is a tree and consider a connected subset $U \subseteq G$ such that $\gamma_o(\partial U)>0$.
Then for all $\m>1$, there is survival in $U$ with positive probability  starting from every $x \in G$.
\end{itemize}
\end{Theorem}

One may ask whether the statements in Theorem~\ref{prop:equivalent} can be strengthened. Namely, 
in $(i)$ we wonder whether $\rho_U=\phi_U=\rho_G$; this is discussed in Remark~\ref{rem:rho=phi} and the answer is negative: indeed, there are examples where $\rho(U)=\phi(U)<\rho(G)$.
From $(ii)$ and $(iii)$ it is natural to conjecture that the extinction probabilities $\{\mathbf{q}(x,U)\}_x$ of BRW in $U$ are related to the properties of $\partial U$.
This is an instance of the more general question of whether $\{\mathbf{q}(x,U)\}_x=\{\mathbf{q}(x,U')\}_x$ is related to $\partial U=\partial U'$ or $\gamma_o(\partial U \triangle \partial U')=0$.
We show that generic 
sets with equal boundary may have different extinction probabilities and there might be survival w.p.p.~in one set and a.s.~extinction in the other
(Proposition \ref{prop:densenopersistence}).
However, in the particular case where $U$ is connected
this cannot happen
 (Propositions \ref{pro:connectedequalboundary} and \ref{pro:survivalpositivemeasure}).
In a nutshell,  we find examples of subgraphs $U$ such that $\gamma_o(\partial U)>0$ and the BRW does not survive in $U$; moreover, there are other cases where $\gamma_o(\partial U)=0$ but the BRW persists globally on $U$ (see Section \ref{subsec:nullandpersistence}).

The following theorem finds precise bounds on $\m$ for global persistence of BRWs in a connected subgraph $U$.

\begin{Theorem}\label{thm:local-surv}
Let $U$ denote a connected subgraph.
Whenever $\m$ is so that
$
\frac{\phi_U}{\rho_U} < \m \leq \frac{1}{\rho_U},
$
then 
the BRW induced by $U$ can persist globally in $U$, but it cannot persist locally.
\end{Theorem}
Note that the statement does \emph{not} imply that under the given conditions there is positive probability of global persistence in $U$.
In order to ensure persistence 
we need some extra assumptions (see below) however, if persistence does occur in this regime, then by Theorem \ref{thm:local-surv} it is only global but not local.

For a fixed subset $U$ and a fixed vertex $x\in U$ define
\begin{equation}\label{eq:defm1}
\m_1(U)=\m_1:=\inf \left  \{\m>1 \ : 
\Pbrw_x \bigl (
\text{ global persistence in }U
\bigr )>0 
\right \}.
\end{equation}
Note that, while in this paper we consider only BRWs defined through a breeding law
with the same distribution at all sites, plus a transition matrix $P$,
in general a BRW is defined by choosing, for each site, a probability distribution on the number of offspring and their position together (and the positions may not be independent of the number of children).
For these general BRWs, the global behavior does not depend only on the first moment (see \cite{Zucca} for details); indeed, it is possible to find two such BRWs with the same first moment and such that one survives globally while the other dies out. Thus, for a general BRW, $\m_1$ might not be a threshold between a null probability of persistence and a positive one.
It is worth noting that even if the BRW $\{{\mathbf B}_n\}_n$ on $G$ is constructed
from a breeding law
with the same distribution $\nu$ at all sites, plus a transition matrix $P$, it is not in general true that $\{{\mathbb U}_n\}_n$ is given by a constant breeding law
$\nu_U$: it may depend on the vertex (while the transition matrix is clearly $P_U$).
Thus it is not trivial to establish when $\m_1$ is a threshold: this is true
if $\nu_U$ does not depend on the vertex (for instance when $P$ is the simple RW and
$U$ is vertex transitive) as a consequence of a result of
\cite{Hutchcroft2020Branching}.
We are able to generalize this statement, and prove that $\m_1(U)$ is still a threshold, (whose precise value can be computed, see Proposition \ref{prop:monotonicity}),  even when the offspring distribution depends on the vertex, provided that $\{{\mathbb U}_n\}_n$ is sufficiently regular, namely in the
case that it is an $\mathcal{F}$-BRW, thanks to Corollary~\ref{thm:local-surv}.
Thus in our case \eqref{eq:defm1} is a well-posed threshold, see Section \ref{sect:proof-thm}.


%
\begin{Corollary}\label{cor:m1+}
Whenever a connected subgraph $U$
is such that the BRW induced by $U$ is an $\mathcal{F}$-BRW, then $\m_1=\frac{\phi_U}{\rho_U}$.
Otherwise, $\frac{\phi_U}{\rho_U}\leq \m_1 \leq \frac{1}{\rho_U}$.
\end{Corollary}



The next result relates the concepts of Martin boundary $\M(Q_U, 1)$ and $t$-Martin boundary
$\M(P_U, t)$ (see Section \ref{sect:construction} for the formal definition of these boundaries).
In \cite{CandelleroHutchcroft} (and, independently, in \cite{KW2022}) the object of study was a ``renormalized version'' of the process $\{\B_n\}_n$ on $G$. 
Similarly, on a fixed subgraph $U$, we rescale the induced BRW by a suitable factor.
For ease of notation, given a subset $U$, define
\begin{equation}\label{eq:bz}
\bz=\bz(U):=\frac{\rho_U}{\phi_U} \le 1,
\end{equation}
and for all $n\in \N$ set
\begin{equation}\label{eq:En}
\mathcal{E}_n(U):= \left \{\text{RW takes }n\text{ consecutive steps in }U \right \}.
\end{equation}
We observe that if $\sum_{y \in U} p(x,y)$ does not depend on $x \in U$ then $\bz=\sum_{y \in U} p(x,y)$ and $\mathbb{P}^x(\mathcal{E}_n(U))=\bz^n$.
Note that $P_U$ is substochastic, but we construct a stochastic matrix $\widetilde P_U$ which rules a RW on $U \cup \{\eth_U\}$, where $\eth_U$ is an absorbing state (the exit from $U$). Namely,
\[
\widetilde{p}_U(x,y):=
 \begin{cases}
  p_U(x,y) & \text{if } x,y \in U \\
  1-\bz & \text{if } x \in U, y=\eth_U\\
  1 & \text{if } x=y=\eth_U\\
  0 & \text{otherwise}.
 \end{cases}
\]

The following result is a version of \cite[Theorem 1.1]{CandelleroHutchcroft} and \cite[Theorems 3.26 and 3.32]{KW2022} (cf.~Theorem~\ref{thm:notes}) for subgraphs.

\begin{Theorem}\label{thm:general-bz-bdary}
Given a nonempty, connected subgraph
$U$,  let $\bz$ be as in \eqref{eq:bz} and $\mathcal{E}_n(U)$ as in \eqref{eq:En}.
Suppose that the RW governed by $P_U$ is transitive.
Then the following hold.
\begin{itemize}
\item[(i)] 
The space $\M(Q_U, 1)$ is homeomorphic to $\M(P_U, \bz)$. 
\item[(ii)] 
Suppose that $\nu$ is such that 
$\Ebrw \left [ L \log L\right ] <\infty$, for $L\sim \nu$.
Almost surely, the process $(\m \bz)^{-n} \mathbb{U}_n$ converges weakly  to a random measure $\mathbf{W}_{\zeta}$ on the space $\M(P_U, \bz)$,
 and it satisfies
\[
\Ebrw_{x} [\mathbf{W}_{\zeta}(A)] =  \Prw_x^{\widetilde P_U}(Y_\infty \in A),\text{ for all Borel sets }A \subseteq \M(P_U, \bz),
\] 
where $\{Y_n\}_n$ denotes a RW governed by $\widetilde P_U$ and $Y_\infty:=\lim_n Y_n$.
\end{itemize}

\end{Theorem}

The paper is organized as follows.
In Section \ref{sec:prelim} we recall some fundamental notions about Martin boundaries and BRWs, and subsequently we define and analyze an auxiliary process that will be helpful throughout this work.
Section \ref{sec:boundarymeasure} is devoted to the discussion of the relationship between the extinction probabilities in subgraphs $U$ and their corresponding boundary.
We provide several examples on trees, Cartesian products and free products.
We conclude with Section \ref{sect:proof} where the reader can find all the proofs of the results claimed above.

\section{Preliminaries}\label{sec:prelim}

In this section, we define the Martin boundary of a RW and recall its relationship with the RW and the BRW (on the whole space). We formally construct the BRW and define survival and persistence. We also recall the definition of $\mathcal F$-BRW, which generalizes the concepts of transitive and quasi-transitive BRW.

\subsection{The Martin boundary}\label{sect:construction}
Here we only recall the fundamental tools that we need in this work, the interested reader is referred to \cite{sawyer_harmonic_functions}, \cite{woess2000} and \cite{Woess-denumerable} for thorough expositions on Martin boundaries.
Let $G$ denote an infinite, locally finite 
graph.
Consider a (sub)stochastic matrix
${P}:=\big ({p}(x,y)\big )_{x,y\in G}$ defined on $G$.
Then for $z\in \mathbb{C}$ let
$
G(x,y\mid z):= \sum_{n\geq 0} {p}^{(n)}(x,y)z^n
$
denote the corresponding Green function; fix a reference vertex $o$ and define the associated Martin kernel by
\[
K(x,y\mid z^{-1}):=\frac{G(x,y\mid z)}{G(o,y\mid z)}.
\]
From now on we shall write 
$G(x,y):= G(x,y|1)$ and $K(x,y) := K(x,y\mid 1)$.
The compactification of $G$ with respect to the convergence of the Martin kernels $K(x,y\mid 1)$
is the Martin boundary $\M({P},1)$.
There are several equivalent ways to construct the Martin boundary, see for example \cite{sawyer_harmonic_functions} and \cite{Woess-denumerable} for details.

Similarly, for a fixed value $t\ge \rho_G$ (as defined in~\eqref{eq:rhoG}), we define the $t$-Martin boundary of ${P}$, that we denote by $\M({P},t)$, as the compactification of $G$ with respect to the convergence of the Martin kernels $K(x,y\mid t)$.
See \cite[Chapter IV]{woess2000} for further details on $t$-Martin boundaries.

\subsection{Martin boundary and branching random walks}\label{sect:constr-propBZ}
In this paper BRWs are defined by the offspring distribution $\nu$ and the transition matrix $P$. This suggests to look for relations between a BRW and the Martin boundary of the RW governed by $P$. We recall the main result in \cite{CandelleroHutchcroft} and \cite{KW2022}, which links the Martin boundary 
of a RW with transition kernel $P$, and the BRW (with same transition kernel) seen as a Markov chain on its space of configurations.

\begin{Theorem}[\cite{CandelleroHutchcroft}, \cite{KW2022}]\label{thm:notes}
Let $Y=\{Y_n\}_n$ be a transient, irreducible Markov chain governed by $P$, let $\nu$ be an offspring distribution with mean $\m>1$ satisfying the \emph{$L \log L$ condition}, and let $\{\B_n\}_{n\geq 0}$ be a BRW with transition kernel $P$ started at some vertex $o$. 
Then $\m^{-n} \B_n$ almost surely converges weakly  to a random measure $\mathbf{W}$ on the Martin compactification $\cM(P,1)$, 
which satisfies
\begin{equation}
\label{eq:ExpectationIdentity}
\Ebrw_{o}[\mathbf{W}(A)]=\Prw^P_o(Y_\infty \in A)
\end{equation}
for every Borel set $A \subseteq \mathcal{M}(P,1)$.
\end{Theorem}

\subsection{Construction of the BRW induced by $U$, survival and persistence.}
At this point we formally define the BRW,
relying on a construction defined in \cite{Zucca}.
Let us consider a family of probability measures $\{\mu_x^{(G)}\}_{x\in G}$ defined on the set of finitely supported functions from $G$ to $\mathbb{N}$
\[
S:=\Bigl \{ f:G \to \N \text{ such that } \sum_{y\in G}f(y)<\infty \Bigr \}.
\]
Subsequently we define the following update rule, to be applied to each $x\in G$. 
A particle at a site $x\in G$ lives for one unit of time, then with probability $\mu_x^{(G)}$ a function $f_x$ is chosen from $S$.
Then, the particle at $x$ is replaced by $f_x(y)$ particles at each site $y\in G$, independently for each $y$.
%
Observe that $\{\mu_x^{(G)}\}_{x \in G}$ can be chosen arbitrarily. By letting $|f|:=\sum_{y\in G}f(y)$ and defining
\begin{equation}\label{eq:defmu}
	\mu_x^{(G)}(f):=\nu(|f|)
	\frac{(|f|)!}{\prod_w f(w)!} 
	\prod_wp(x,w)^{f(w)}
\end{equation}
we have the process $\{\B_n\}_n $.

We construct the \emph{process induced by} $U\subseteq G $,  which we denote by $\{\mathbb{U}_n(x)\}_n$, as follows.
Clearly, whenever $U\equiv G$ then this construction gives $\{\B_n\}_n$.
More precisely, consider a family $\{f_{i,n,x}\}_{i,n\in \N,x\in G}$ of independent $S$-valued random variables such that for each $x\in G$ the law of $\{f_{i,n,x}\}_{i,n\in \N}$ is $\mu_x^{(G)}$.
%

For all $n\geq 0$ and all $x,w\in U$ we denote by  $\mathbb{U}_n^x(w) $ the number of particles of the auxiliary BRW started at $x$ that are alive at time $n$ at vertex $w$.
If we do not need to keep track of the starting vertex we simply write $\mathbb{U}_n(w) $ for the number of particles alive at $w$ at time $n$.
The initial condition is $\mathbb{U}_0^x(w)=\delta_x(w) $ and 
$\mathbb{U}_n^x(\cdot) $
is defined inductively as
\begin{equation}\label{eq:defBZnxy}
\mathbb{U}_{n+1}^x(w):=\sum_{y\in U} \sum_{i=1}^{\mathbb{U}_{n}^x(y)} f_{i,n,y}(w)= \sum_{y\in U}\sum_{j\geq 0} \mathbf{1}_{\{\mathbb{U}_{n}^x(y)=j\}} \sum_{i=1}^j f_{i,n,y}(w) .
\end{equation}
%
%
Then the expected number of particles sent to vertex $w\in U$ from one particle at $x\in U$  in one unit of time is defined by
\begin{equation}\label{eq:firstmoment}
m_{x,w}:=\sum_{f \in S} f(w) \mu_x^{(G)} (f),
\end{equation}
and the matrix $M:=(m_{x,w})_{x,w \in U}$ is the matrix whose $x,y$-entry is $m_{x,y}$, for all $x,y\in U$.
For $n\in \N$ consider the power $M^n$, and for all $x,w \in U$ let $m_{x,w}^{(n)}$ be their corresponding entry in $M^n$.
For every $n\geq 0$ and all $x,w\in U$ we have
\begin{equation}\label{eq:expectedBZn}
\Ebrw \left [ \mathbb{U}_n^x(w)  \right ] =m_{x,w}^{(n)}.
\end{equation}
Alternatively, the BRW induced by $U$ can be seen as follows.
Define the set
$
S_U:=\{f:U\to\N \, : \, \sum_{y\in U}f(y)<\infty\}
$
and for now let $|f|=\sum_{y\in U}f(y)$ for all $f\in S_U$. 
For all $x \in U$ let $\mu_x^{(U)}$ be a probability measure on $S_U$, such as
\begin{equation}\label{eq:muxinduced}
\mu_x^{(U)}(f):=\sum_{n=|f|}^\infty \nu(n)\left(1-\sum_{w\in U}p_U(x,w)\right)^{n-|f|}
\frac{n!}{(n-|f|)!\prod_w f(w)!}
\prod_wp_U(x,w)^{f(w)}.
\end{equation}
This quantity represents the probability that an individual at site $x$ has exactly $f(y)$ children at site $y$, for all $y\in U$.
The two families $\{\mu_x^{(G)}\}_x$  and $\{\mu_x^{(U)}\}_x$ are 
such that for all $f \in S_U$ and $ x \in U$
\[
\mu^{(U)}_x (f) = \mu^{(G)}_x \left ( \tilde{f} \in S_G \ : \ \tilde{f}\bigr |_U = f\right ) .
\]

\begin{Definition}\label{def:global-local}
We call \emph{global persistence in} $U$ of 
the BRW  $\{\B_n\}_n$ 
the event
$
\{ \sum_{w\in U}\mathbb{U}_{n}(w)>0 , ~ \forall \, n \geq 1 \}
$.
Analogously, we call \emph{local persistence} at $x\in U$ 
the event
$
\{ \limsup_n \mathbb{U}_{n}(x ) >0  \}.
$
When we say that there is \emph{local persistence} without specifying the vertex, we mean that there is local persistence at every vertex.
\end{Definition}
If not otherwise specified, when we say that \emph{the process persists (locally/globally)} or that \emph{there is (local/global) persistence}, we tacitly imply that the corresponding event has positive probability.

Roughly speaking, global persistence of $\{\B_n\}_n $ on $U$ means that the induced BRW $\{ \mathbb{U}_{n}^x\}_n$ has some particles somewhere in $U$ for all $n\geq 0$.
Local persistence of $\{\B_n\}_n $ in $U$ means that there are infinitely many times $n\geq 0$ such that at least a particle of $\{ \mathbb{U}_{n}\}_n$ is at $x$. 
\begin{Definition}\label{def:survival-U}
Given a BRW $\{\B_n\}_n $, we call \emph{survival} in $U$ 
the event
$
\{ \limsup_n \sum_{x\in U}\B_{n}(x)>0 \}.
$ If $U=G$ we speak of \emph{global survival}; while if $U=\{x\}$ is a singleton, we speak of \emph{local survival at $x$} (or, simply, \emph{local survival} if it holds for every singleton). \emph{Extinction} in $U$ is the complement event, that is, when there are a finite number of visits in $U$.
\end{Definition}
If not otherwise specified, when we say that \emph{the process survives (locally/globally)} or that \emph{there is (local/global) survival}, we tacitly imply that the corresponding event has positive probability. Conversely, when we say that \emph{the process dies out}, we mean that the extinction event has probability 1.

Clearly, since $\sum_{w\in U}\mathbb{U}_{n}^x(w) \leq \sum_{w\in U}\B_{n}(w) $ uniformly in $x$, then global persistence in $U$ implies global survival in $U$.
%
%
%

\subsection{$\mathcal{F}$-BRWs.}
Now we introduce the fundamental concept of $\mathcal{F}$-BRW, relying on \cite{Zucca} and \cite{bertacchi-zucca-infinite-type}.
\begin{Definition}\label{def:F-BRW}
We say that a BRW with state space $X$ and reproduction rule governed by $\{\mu_x\}_{x\in X} $ is \emph{projected} on a BRW with state space $Y$ and reproduction rule governed by $\{\tilde{\mu}_y\}_{y\in Y} $ if there exists a surjective map $g: X \to Y$ such that
\[
\forall \, 
f\in S_{Y}, \quad  \tilde{\mu}_{g(x)} ({f}) = \mu_x \Bigl \{ h \in S_X \, : \, \forall y \in Y , \ {f} (y) = \sum_{z \in g^{-1}(y)}h(z) \Bigr \}.
\]
The map $g$ is called the \emph{projection map}.
Furthermore, whenever $ Y$ is finite, we will say that the BRW with state space $X$ and reproduction rule governed by $\{\mu_x\}_{x\in X} $ is 
an $\mathcal{F}$-BRW.
%
\end{Definition}


The fact that $\{\mathbb{U}_n\}_n$ is an $\mathcal{F}$-BRW does not depend on the offspring distribution $\nu$, as we show in the next proposition which is a corollary of a more general result (see Proposition~\ref{prop:fbrwgen} in Section~\ref{sect:proofFBRW}).
\begin{Proposition}\label{prop:fbrw}
The following are equivalent.
\begin{enumerate}[(a)]
\item
The BRW induced by $U$ is an $\mathcal{F}$-BRW for all offspring distributions $\nu$.
\item
The BRW induced by $U$ is an $\mathcal{F}$-BRW for some offspring distribution $\nu_0$ such that $\nu_0(0)<1$.
\item
There exists a surjective map $g:U\to Y$, with $Y$ finite set, such that the quantity $\sum_{w\colon g(w)=y}p_U(x,w)$ only depends on $g(x)$ and $y$.
\end{enumerate}
\end{Proposition}

In the case when
$\mathbb U_n$ is an $\mathcal F$-BRW,
we are able to
give an explicit expression
for $\m_1(U)$.

\begin{Proposition}\label{prop:monotonicity}
Suppose that the BRW induced by $U$ is an $\mathcal F$-BRW, then there is
positive probability of global persistence on $U$ starting from $x$ if and only if 
\[
\mathbf{m}>\frac{1}{\liminf_{n\to\infty}\sqrt[n]{\sum_{y\in U}p_U^{(n)}(x,y)}}=\m_1(U).
\]
In particular if there is global persistence in $U$, starting from $x$, for the process with an offspring law $\nu_1$ with mean $\bar\nu_1$, then there is global persistence in $U$, starting from $x$, for
any process with offspring law $\nu_2$ with mean $\bar\nu_2 \ge \bar\nu_1$.
\end{Proposition}

\section{Survival, persistence and boundary measure}
\label{sec:boundarymeasure}

In this section we answer several questions on the relationship between the Martin
boundary and the concepts of survival and persistence. To this aim we construct
various examples, mainly on trees, which are a natural setting since on the one hand, the simple RW on trees gives rise to Martin boundaries with a geometrical construction, and on the other hand, in many cases, $\rho_G<1$,  which implies that the RW is transient and, when $1 <\m \le 1/\rho_G$, the BRW is transient as well.
In Section~\ref{subsec:topologytree} we focus on the topology of the Martin boundary of a nearest neighbor RW $P$ on a tree.
In Sections~\ref{subsec:nullandpersistence} and \ref{subsec:densenosurv} we provide examples showing that the natural intuition that a BRW and the underlying RW might show very similar asymptotic behaviors is misleading.
More precisely, in Section~\ref{subsec:nullandpersistence} we find examples of subgraphs $U$ whose boundaries have harmonic measure $0$, nevertheless the BRW persists in $U$. Conversely, in Section~\ref{subsec:densenosurv} we find examples of subgraphs $U$ whose boundaries have positive harmonic measure, but the BRW neither survives nor persists in $U$. Furthermore, in Section~\ref{subsec:densenosurv} we investigate the relationship between some properties of the boundary of a subset $U$ and the extinction probabilities of the BRW in $U$.
In Section~\ref{subsec:tosumup} we deal with the relationship between extinction probabilities in subset and the boundary of the subset.

Unless otherwise specified, in this section we consider 
\textit{edge-breeding} BRWs, which are defined as follows.
Given $\lambda>0$, every particle at $x\in G$ has $n$ children ($n \in \mathbb{N}$) with probability $\nu_x(n):=(\lambda d_x)^n /(1+\lambda d_x)^{n+1}$, where $x$ has exactly $d_x$ neighbors. Each child is placed independently and uniformly on the neighbors.
This is the discrete-time counterpart of the continuous-time edge-breeding BRW on $G$. The offspring law at $x$ is a geometric distribution with parameter $\lambda d_x$ and $P$ is the transition kernel of the simple RW on $G$.
In particular, when $G$ is the homogeneous tree $\mathbb{T}_d$, then $\nu$ does not depend on $x$  and when $\lambda \in (1/d, \, 1/2\sqrt{d-1}]$ there is global survival w.p.p~and a.s.~local extinction; unless told otherwise, this will be our choice in the rest of this section every time we deal with a homogeneous tree.
These critical parameters can be deduced from \cite[Proposition 4.33]{Bertacchi-Zucca-book} and in this case
$\m=\lambda d \in(1,1/\rho_G]$.
Note that the restriction of an edge-breeding BRW to a subgraph is the edge-breeding BRW on that subgraph; thus if we consider 
$G=\mathbb T_d$ and $U\subseteq \mathbb T_d$, the breeding law $\nu$ is constant on $G$ and is given by the above expression
$\nu_x$ on $U$, provided that $d_x$ identifies the number of neighbors of $x$ in $U$.

\subsection{The topology and the Martin Boundary of a tree}\label{subsec:topologytree}

Given a tree $\mathbb{T}$, let $P$ be a nearest neighbor transition matrix
 (take for instance the transition matrix $P$ of a simple RW).
We fix a reference vertex $o\in\mathbb T$ and call it the root of the tree.
Paths in the tree are maps  $\varphi \in \mathbb{T}^\mathbb{N}$ such that  $(\varphi(i), \varphi(i+1))$ is an edge for all $i \in \mathbb{N}$. With a slight abuse of notation, we also write $\varphi \subseteq A\subseteq\mathbb T$,
 meaning that we identify $\varphi $ with its image.
We denote by $\varphi_{x,y}$ the shortest path from $x$ to $y$ and by
$d(x,y)$ its length (that is, the distance between $x$ and $y$); the shortest path from $x$ to $y$ in a tree is the only path from $x$ to $y$ without repeated vertices.
Sometimes, it is also referred to as the \emph{geodesic} path.

The Martin boundary of $\mathbb{T}$ can equivalently (up to homeomorphisms) be defined as the collection of all the infinite injective paths from $o$ (thus $\varphi(0)=o$). These paths are called (geodesic) \textit{rays}.
This construction does not depend (up to homeomorphisms) on the choice of the root $o$ (see \cite[Chapter 9]{Woess-denumerable} and references therein). 
Since every vertex $x$ in $\mathbb{T}$ can be identified with $\varphi_{o,x}$ we have that the Martin compactification is (up to homeomorphisms) the collection of all finite or infinite injective paths from $o$. 
We denote the Martin compactification of $\mathbb T$ by $\widehat{ \mathbb{T}}= \mathbb{T} \cup \mathcal{M}(P,1)$.

Since our goal is to describe the behavior of the BRW in sets $U$ with a certain boundary, or in sets which are neighbors of fixed subsets
of $\mathcal{M}(P,1)$, we need to understand the topology of $\widehat{ \mathbb{T}}$.

 Given a subset $A \subseteq \widehat{\mathbb{T}}$ we denote by $\widehat A$ its topological closure and by $\partial A$
its boundary
(which is the intersection of its closure and the closure of its complement).
It is clear that $\partial \mathbb T=\mathcal M(P,1)$ and 
$\partial A = \widehat A \cap \mathcal{M}(P,1)$.  Moreover, $\partial A$ is always closed in $\mathcal{M}(P,1)$ with the induced topology; conversely, any closed subset of $\mathcal{M}(P,1)$ is the boundary of a subset of $\mathbb{T}$ (see Lemma~\ref{lem:boundaryA}).

Define $\mathbf{T}_x \subseteq \mathbb{T}$ as the subtree given by the vertices $y$ such that $x$ belongs to the path $\varphi_{o,y}$. 
A countable base for the topology of $ \widehat{\mathbb{T}}$  is the collection 
$\{\{x\} \colon x \in \mathbb{T} \}
\cup \{\widehat{\mathbf{T}_x} \colon x \in \mathbb{T} \}$.
As a consequence, a countable base for the induced topology on $\mathcal{M}(P,1)$ is the family $\{\partial \mathbf{T}_x \colon x \in \mathbb{T} \}$.
Given two basic open sets 
$A$ and $B$, then either $A \subseteq B$ or $B \subseteq A$ or $A \cap B=\emptyset$; therefore, every open set of $\widehat{\mathbb{T}}$ (resp.~$\partial \mathbb{T}$) is a countable union of pairwise disjoint elements of the base.
Note that $\widehat{\mathbf{T}_x}$
(resp.~$\partial \mathbf{T}_x$) and the elements of the base are both open and closed.
A generic open neighborhood of $\partial \mathbb{T}$ is $\widehat{\mathbb{T}}\setminus A$ where $A$ is a finite subset of $\mathbb{T}$.
Moreover, if $A$ is such that, for every $x \in \mathbb{T}$, there exists $y \in A$ such that $y \in \mathbf{T}_x$, then $\partial A= \partial \mathbb{T}$. 

The subtree $\mathbf{T}_x \subseteq \mathbb{T}$ branching from $x$ is the set of vertices $y$ such that $x$ belongs to the path $\varphi_{o,y}$. 
We denote the Martin compactification by $\widehat{ \mathbb{T}}= \mathbb{T} \cup \mathcal{M}(P,1)$.
Given a subset $A \subseteq \widehat{\mathbb{T}}$ we denote by $\widehat A$ its topological closure and by $\partial A = \widehat A \cap \widehat {A^\complement}$ its boundary (where the complement is taken in the set $\widehat{\mathbb{T}}$). Since the induced topology on $\mathbb{T}$ is the discrete topology, then it is clear that $\partial A = \widehat A \cap \mathcal{M}(P,1)$.  Clearly $\partial A$ is always closed in $\mathcal{M}(P,1)$ with the induced topology; conversely, any closed subset of $\mathcal{M}(P,1)$ is the boundary of a subset of $\mathbb{T}$ (see Lemma~\ref{lem:boundaryA}).

\subsection{Null measure at the boundary and persistence}\label{subsec:nullandpersistence}

In this section we find examples of infinite sets $U$ where the BRW 
persists in a subset $U$,
but the RW has zero probability of staying there forever.
In other words, there are infinitely (uncountably) many trajectories of BRW that are not typical trajectories of the underlying RW.


\begin{Example}[\textbf{GW tree}]
\label{exmp:GWtree}
Consider a tree $\mathbb{T}_{\{n_i\}_i}$ where every vertex at distance $i$ from the root $o$ has $n_i \ge 1$ neighbors at distance $i+1$. Let $S_i$ denote the set of vertices at distance $i$ from $o$; clearly $|S_i| = \prod_{j=0}^{i-1} n_j$. 

Suppose that the simple RW on $\mathbb{T}_{\{n_i\}_i}$ is transient and let $\widetilde \gamma_o$ be
 the associated measure on the Martin boundary  $\partial \mathbb{T}_{\{n_i\}_i}$.
In this case the value of $\widetilde \gamma_o$ on the borders of subtrees branching from fixed vertices is
 $\widetilde \gamma_o(\partial \mathbf{T}_x) =1/|S_i|$ for every $x \in S_i$.
\\
Fix $p  \in [0,1]$ and denote by $\Upsilon$ the GW tree with root $o$ identified by the connected component (containing $o$) of a 
$p$-Bernoulli percolation on 
$\mathbb{T}_{\{n_i\}_i}$. Clearly, if $x \in S_i$, then the probability that $x \in \Upsilon$ is $p^j$. Suppose that $p$ is sufficiently
 large so that the percolation cluster is infinite with positive probability.
\\
It is easy to show, by induction on $j$, that $\widetilde{\mathbb{E}}[\widetilde \gamma_o(\partial \Upsilon)] \le p^j$ (where $\widetilde{\mathbb{E}}$ is the expectation with respect to the GW distribution). This implies that $\mathbb{E}[\widetilde \gamma_o(\partial \Upsilon)] =0$, therefore $\widetilde \gamma_o(\partial \Upsilon)=0$ for almost every realization of the GW tree. 
We observe that the same result holds if we open edges from  $S_i$ to $S_{i+1}$  independently with probability $p_i$ where  $\sum_{i=0}^\infty (1-p_i)=+\infty$.

Now take $n_i:=d-1$, $d \ge 3$.
Then $\mathbb{T}_{\{n_i\}_i}$ is a subtree of the homogeneous tree $\mathbb{T}_d$ (the root $o$ has $d-1$ neighbors, all other vertices have $d$ neighbors). Denote by $\gamma_o$ the harmonic measure associated to the simple RW on $\mathbb{T}_d$; clearly $\gamma_o(\partial A)= (1-1/d) \widetilde \gamma_o(\partial A)$ for all $A \subseteq \mathbb{T}_{\{n_i\}_i}$ (where $\widetilde \gamma_o$ is the measure on $\partial \mathbb{T}_{\{n_i\}_i}$ defined above). We consider the edge-breeding BRW on $\mathbb{T}_d$ (which is an $\mathcal{F}$-BRW). 
It is easy to prove, by a coupling argument, that if $\lambda \le 1/2\sqrt{d-1}$ then almost surely there is no local persistence at any $x$ for almost every realization of $\Upsilon$.
Furthermore, according to \cite[Proposition 2.6]{PemantleStacey}, if $\lambda \ge 1/(1+(d-1)p)$ then there is global persistence with positive probability on almost every infinite realization of $\Upsilon$ (by using the notation of \cite[Proposition 2.6]{PemantleStacey}, the value of $m$ in our case is $(d-1)p$). Thus if $(2 \sqrt{d-1}-1)/(d-1) < p <1$ and $\lambda \in [ 1/(1+(d-1)p),\,  1/2\sqrt{d-1}]$ then 
 $\Upsilon$ is infinite with positive probability and, in almost every infinite  realization of $\Upsilon$, 
there is global persistence (hence survival) w.p.p.;  nevertheless, $\widetilde \gamma_o(\partial \Upsilon)=0$, that is $\gamma_o(\partial \Upsilon)=0$ (since $\Upsilon \subseteq \mathbb{T}_{\{n_i\}_i}$). Note that every realization of $\Upsilon$ is connected and its boundary $\partial \Upsilon$ is almost surely nowhere dense.
\end{Example}

While the above example applies to $\mathbb{T}_d$ for every $d \ge 3$, there is a deterministic example of subtree in $\mathbb{T}_d$, which has zero-measure boundary, and 
 where the BRW persists globally when $d \ge 6$.

\begin{Example}[\textbf{Pruning a homogeneous tree}]
\label{exmp:pruningtree}
Given a strictly increasing sequence of integers $\{k_i\}_{i \in \mathbb{N}}$ and a homogeneous tree $\mathbb{T}_d$,
consider the following  construction.
Recursively remove one edge (and the corresponding subtree) going from $x \in S_{k_i}$ to $S_{k_i+1}$ for every $x \in S_{k_i}$; let us call $\mathbb{T}'$ the resulting pruned tree. As in the previous example, denote by $\gamma_o$ the harmonic measure associated to the simple RW on $\mathbb{T}_d$. By using a similar argument as in the case of the GW trees, one can prove by induction that $\gamma_o(\partial \mathbb{T}') \le (d-1)^i/d^i$ for all $i \in \mathbb{N}$, since at every step when we remove one edge from each vertex on $S_{k_i}$ we reduce the boundary by a multiplicative factor $(1-1/d)$; thus $\gamma_o(\partial \mathbb{T}')=0$. 
Consider now the tree $\mathbb{T}_{d-1}$ as a subtree of $\mathbb{T}_d$; it can be obtained by the above construction when $k_i=i$ for all $i \in \mathbb{N}$. 
As as subset of $\mathbb{T}_{d}$ we then have
$\gamma_o(\partial \mathbb{T}_{d-1})=0$. We know that if $1/2\sqrt{d-1} \ge \lambda > 1/(d-1)$ then the edge-breeding BRW on $\mathbb{T}_d$   persists globally on $\mathbb{T}_{d-1}$ and dies out locally on $\mathbb{T}_d$; the existence of such a $\lambda$ is guaranteed when $d \ge 6$. Again, the subset in this example is a connected subtree and its boundary is nowhere dense.
\end{Example}

Another example is given by the cartesian product of two trees $\mathbb T_{k}$ and $\mathbb T_{j}$: in the product, each component has zero-measure boundary, but for sufficiently large $\m$, one might have positive probability of persistence in the single component.
The particular choice $k=3$, $j=100$ guarantees the existence of a range for $\m$ such that the BRW survives globally but 
not locally (and where we have positive probability of persistence in at least one of the components).

\begin{Example}[\textbf{Cartesian product of different trees}]
\label{exmp:cartesian}
Consider the two trees $\T_3$ and $ \T_{100}$.
For $j\in \{3,100\}$ let $\partial \T_j$ denote the end compactification of $\T_j$, and define $\widehat{\T}_j:=\T_j \cup \partial \T_j $.
Take $G:=\T_3\times \T_{100}$ and consider 
 $U_1=\widehat{\T}_3\times o_{\T_{100}}$ and $U_2=o_{\T_3} \times \widehat{\T}_{100}$.
Let $P:=\frac{3}{103 }P_1 + \frac{100}{103} P_2$, where $P_i$ denotes the transition kernel of the simple RW on $U_i$.

With our notation we obtain $\rho_{U_1}=\frac{3}{103}\phi_{U_1}$ and $\rho_{U_2}=\frac{100}{103}\phi_{U_2}$, where $\phi_{U_1}=\frac{2\sqrt{2}}{3}$ and $\phi_{U_2}=\frac{2\sqrt{99}}{100}$, and $\rho_G=\frac{3}{103}\phi_{U_1}+\frac{100}{103}\phi_{U_2}$.
Thus, $\rho_G^{-1}\approx 4.5$, and
$\frac{\phi_{U_1}}{\rho_{U_1}}= \frac{103}{3}$ and $\frac{\phi_{U_2}}{\rho_{U_2}}= \frac{103}{100}=1.03<\rho_G^{-1}$.
By what said above, the range $\phi_{U_2}/\rho_{U_2}<m\le \rho_G^{-1}$
is non-empty.
Since the BRW induced by $U_2$ is an $\mathcal{F}$-BRW, by Lemma \ref{lemma:step3} we have $\m_1(U_2)=\phi_{U_2}/\rho_{U_2}$.
As a consequence, for $\m > \m_1(U_2)$, we have positive probability of persistence in $U_2$ and in this regime $\{\B_n\}_n$ is still transient.

It is known from \cite{picardello_woess} that by letting $\mathcal{S}_1:=\{(t_1,t_2)\in [\phi_{U_1}, \infty )\times [\phi_{U_2}, \infty) \  : \ \frac{3}{103}t_1+\frac{100}{103}t_2=1 \}$, one has
\[
\M(P,1)= \left ( \partial \T_3 \times \partial \T_{100} \times \mathcal{S}_1 \right ) \cup \left ( \partial \T_3 \times \T_{100} \right ) \cup \left ( \T_3 \times \partial \T_{100} \right ),
\]
with $\gamma_o(o_{\T_3} \times \partial \T_{100})= \gamma_o(\partial U_2)=0$.
\end{Example}

Finally, one can see the same phenomenon in BRWs on the Cayley graph of a free product of finitely generated groups.

\begin{Example}[\textbf{Free product of groups}]
\label{exmp:freeproduct}
In this context we recall the basics of the construction of a \emph{free product of two groups}, the interested reader is referred to \cite{woess2000} for more details.
Consider two finitely generated groups $\Gamma_1, \Gamma_2$ (so that $|\Gamma_1| \geq 2, |\Gamma_2| >2$), with identity elements $e_1, e_2$ respectively, and let $S_i$ denote the set of generators of $\Gamma_i$, $i\in \{ 1,2\}$.
We see elements of $S_i$ as ``letters'' and concatenation of letters as ``words''; we only consider \emph{reduced} words, i.e., without cancellations such as $aa^{-1}$.

The free product $\Gamma:=\Gamma_1 \ast \Gamma_2$ is the set of all (reduced) words of the form $x_1 x_2 \cdots x_n$, where $x_1, x_2, \ldots, x_n \in \bigcup_{i=1}^2 \Gamma_i\setminus \{e_i\}$ and so that $x_i\in \Gamma_1 ~(\text{resp.\ }\Gamma_2)~ \Rightarrow x_{i+1}\in \Gamma_2~(\text{resp.\ }\Gamma_1)~$ for all $i\in \{1 , \ldots , n-1\}$.
An intuitive way to visualize $\Gamma$ is as follows: consider a copy of $\Gamma_1$ and to each vertex $v\in \Gamma_1\setminus \{e_1\}$ attach a copy of $\Gamma_2$ by ``gluing'' (i.e., identifying) $e_2$ and $v$ into a single vertex.
Then, inductively, for every vertex on each copy of $\Gamma_2\setminus \{e_2\}$ attach a copy of $\Gamma_1$, for every vertex on each copy of $\Gamma_1\setminus \{e_1\}$ attach a copy of $\Gamma_2$ and so on.
This construction gives rise to  a ``\emph{cactus-like}'' structure, which whenever all factors $\{\Gamma_i\}_{i=1}^2$ are finite, turns into a \emph{tree-like} structure.
 In the latter case, $\Gamma$ is hyperbolic.

Let $\Gamma_1$ be a finitely-generated group and $\Gamma_2$ be an \emph{infinite}, finitely-generated, \emph{non-amenable} group (not a tree, because trees make this discussion less interesting).
The obtained structure is still a group, and it is non-amenable (see e.g.\
\cite{woess2000}), its identity element is identified with $e_1$ and $e_2$ and simply denoted by $e$.
Let $\mu_1$ and $\mu_2$ denote probability measures defined on the set of generators $S_1$ and $S_2$ respectively, so that they induce (irreducible) nearest neighbor RWs on $\Gamma_1$ and $\Gamma_2$ respectively.
Fix $\alpha\in (0,1)$ and let
$
\mu:=\alpha \mu_1 + (1-\alpha)\mu_2
$
define a nearest neighbor RW 
on $\Gamma=\Gamma_1\ast \Gamma_2$.

Consider $U$ to be a fixed copy of $\Gamma_2$.
From the definition of the process and Corollary \ref{cor:m1+}
(using vertex-transitivity of $U$), we have that
\[
\m_1=\left ( \Prw_e(\mathcal{E}_1(U))\right )^{-1}
=(1-\alpha)^{-1}.
\]
It is known (cf.\ \cite{woess2000} and references therein) that $\gamma_e(\partial U)=0$, however, whenever $\m > (1-\alpha)^{-1}$ there is positive probability of persistence in $U$.
\end{Example}

\subsection{Is the boundary of a set related to survival of BRW?}
\label{subsec:densenosurv}

In the previous section we showed that even if the boundary is ``small'' then we can have persistence.
At this point, one can wonder whether ``large'' sets imply persistence, or at least survival for BRW.
In this section we find examples where the answer is negative.
More precisely, one may have sets with different boundaries and equal extinction probability vectors 
or sets with equal boundary and different extinction probability vectors 
(see Proposition \ref{prop:densenopersistence} and Remark~\ref{rem:densenopersistence}, respectively).

Recall that $\mathbf{q}(x, A)$ is the probability of extinction in $A$ starting with one particle at $x$. We denote by $\mathbf{q}(A)$ the vector $(\mathbf{q}(x,A))_{x \in G}$; note that $\mathbf{q}(x,\emptyset)=1$ for all $x\in G$.

The following proposition tells us that, given any subset $B$ of $\mathbb{T}_d$,
there exists another subset  $A_B$ whose boundary is $\partial \mathbb{T}_d$ with the same extinction probabilities as $B$.
\begin{Proposition}\label{prop:densenopersistence}
Consider the edge-breeding BRW on $\mathbb{T}_d$ with $1/d < \lambda \le 1/2\sqrt{d-1}$.
 Let $B\subseteq \mathbb{T}_d$ (including $B=\emptyset$) then there exists $A_B \subseteq \mathbb{T}_d$ such that $\partial A_B=\partial \mathbb{T}_d$ and $\mathbf{q}(B)=\mathbf{q}(A_B)$.
\end{Proposition}

\begin{Remark}\label{rem:densenopersistence}
   Proposition~\ref{prop:densenopersistence} shows that there are sets with different boundaries and equal extinction probabilities. However the same proposition yields the existence of sets with equal boundaries and different extinction probabilities.
   Indeed let $B_1, B_2 \subseteq \mathbb{T}_d$ such that $\mathbf{q}(B_1) \neq \mathbf{q}(B_2)$. Consider $A_{B_1}$ and $A_{B_2}$; $\mathbf{q}(A_{B_1}) \neq \mathbf{q}(A_{B_2})$ but $\partial A_{B_1}=\partial \mathbb{T}_d= \partial A_{B_2}$.
\end{Remark}

It turns out (see Section \ref{sect:proofdensenopersistence}) that the set $A_B$ might not be connected. 
If $A$ and $B$ are connected subsets of a generic tree, then having equal boundaries implies equal extinction probability vectors as the following proposition shows.
A partial converse is given by Proposition \ref{pro:survivalpositivemeasure}.

\begin{Proposition}\label{pro:connectedequalboundary}
Consider a tree $\mathbb{T}$ and a BRW with a nearest neighbor transition kernel.
If $A \subseteq  \mathbb{T}$, then there exists a connected subset $B \subseteq \mathbb{T}$ such that $\partial A=\partial B$.
Moreover, if $A$ and $B$ are connected subsets of $\mathbb{T}$ and $\partial A = \partial B$ then $\mathbf{q}(A)=\mathbf{q}(B)$. 
\end{Proposition}

The set $A_\emptyset$ in Proposition~\ref{prop:densenopersistence} is an example of a set with very large boundary, namely $\partial \mathbb{T}_d$, and a.s.~extinction ($\mathbf{q}(A_\emptyset)=\mathbf 1$).
It is worth pointing out that
in a connected subset of $\mathbb{T}$ whose boundary contains an open set (in the topology induced on the boundary) there is trivially positive probability of persistence.
Indeed if a set $A$ is connected and dense in an open subset of the boundary, it is dense in $\partial \mathbf{T}_x$ for some $x$. Whence, for all $y \in \mathbf{T}_x$ there exists $a_y \in \mathbf{T}_y \cap A$. Since $A$ is connected, then all vertices $\{a_y\}_{y \in \mathbf{T}_x}$ are connected by a path inside $A$, thus $x \in A$ and $y \in A$ for all $y \in \mathbf{T}_x$. Since $A \supseteq \mathbf{T}_x$ and there is positive probability of persistence in $\mathbf{T}_x$ then there is positive probability of persistence in $A$.
One can wonder if any closed set on the boundary with positive measure contains an open set.
Surprisingly, the answer is negative; indeed, according to the following remark, a nowhere dense subset of the boundary can have a measure arbitrarily close to 1.

\begin{Remark}[\textbf{Open dense subset with arbitrarily small measure}]
\label{rem:opendensesmall}

It is not difficult to see that a subset of the boundary of a tree $\mathbb{T}$ is open and dense if and only if it
is equal to $\bigcup_{y \in A} \partial \mathbf{T}_y$ for some $A \subseteq \mathbb{T}$ with the property that for all $x \in \mathbb{T}$ there exists $y \in A$ such that either $x \in \mathbf{T}_y$ or $y \in \mathbf{T}_x$. 

Consider now the tree $\mathbb{T}:=\mathbb{T}_{\{n_i\}_i}$ and the measure $\gamma_o$ as defined above. 
Suppose that $n_i \ge 2$ for infinitely many $i \in \mathbb{N}$. 
 Let us fix $\epsilon >0$ and let $\{\bar n_i\}_{i \in \mathbb{N}}$ be such that $\sum_{i \in \mathbb{N}} 1/|S_{\bar n_i}| \le \epsilon$. Consider an arbitrary ordering $\{x_i\}_{i\in \mathbb{N}}$ of the vertices of $\mathbb{T}$; let us fix a sequence $\{y_i\}_{i \in \mathbb{N}}$ of vertices such that $y_i \in \mathbf{T}_{x_i}$ and $d(0, y_i) \ge \bar n_i$. The set 
$B:=\bigcup_{i \in \mathbb{N}} \partial \mathbf{T}_{y_i}$ is open and dense and $\gamma_o(B) \le \sum_{i \in \mathbb{N}} \gamma_o(\partial \mathbf{T}_{y_i}  )
\le \sum_{i \in \mathbb{N}} 1/|S_{\bar n_i}| \le \epsilon$. Clearly if $A:=\bigcup_{i \in \mathbb{N}}\{y_i\}$ then $\partial A=\partial \mathbb{T}$.
This means that there exists an open and dense subset in the boundary with $\gamma_o$-measure arbitrarily close to $0$. By taking the complement set, we obtain a closed, nowhere dense subset in the boundary with $\gamma_o$-measure arbitrarily close to $1$.
\end{Remark}

For a generic connected subset whose boundary has  positive measure, positive probability of survival is guaranteed by Theorem~\ref{prop:equivalent}(iii)
(generalized by the following proposition),
%
%
where, we show that whenever $A,B \subseteq \mathbb{T}$ are two connected sets such that $\gamma_o(A \triangle B )>0$, then $\mathbf{q}(A)\neq\mathbf{q}(B)$.

\begin{Proposition}\label{pro:survivalpositivemeasure}
Let us consider a BRW on a tree $\mathbb{T}$ which survives globally but not locally.
Let $A, B \subseteq \mathbb{T}$ be two connected subsets such that $\gamma_o(\partial A \triangle \partial B)>0$. Then $\mathbf{q}(A) \neq \mathbf{q}(B)$.
\end{Proposition}

\subsection{Boundaries and extinction probabilities: a complete description}\label{subsec:tosumup}

With the aid of the examples and results of Sections~\ref{subsec:nullandpersistence} and~\ref{subsec:densenosurv} we are able to look into a question raised in some recent papers (see e.g.~\cite{BZ2020,BertacchiZuccaAndFriends}), namely what conditions on the subsets $A$ and $B$ imply or are implied by $\mathbf{q}(A) = \mathbf{q}(B)$, and when one has that $\mathbf{q}(A) = \mathbf{1}$.
By looking at the approach in \cite{BZ2020} on trees, one is tempted to look for conditions on the boundaries $\partial A $ and $\partial B$.
The examples shown in this paper suggest that there is no such a connection. 
Consider the following tables. 
Here implications are denoted by an arrow; if an implication does not hold we draw a slash on the arrow.
When $A$ and $B$ are \emph{connected}, we use red, single arrows;
otherwise we use black, double arrows.
 Clearly $\textcolor{red}{\nrightarrow}$ is stronger than $\nRightarrow$ and $\Rightarrow$ is stronger than $\textcolor{red}{\rightarrow}$.
Analogously, $\Rightarrow$ (resp.~$\textcolor{red}{\rightarrow}$) on the left table are stronger that the corresponding implication on the right one and $\nRightarrow$ (resp.~$\textcolor{red}{\nrightarrow}$) on the right table are stronger than on the left one.
\medskip


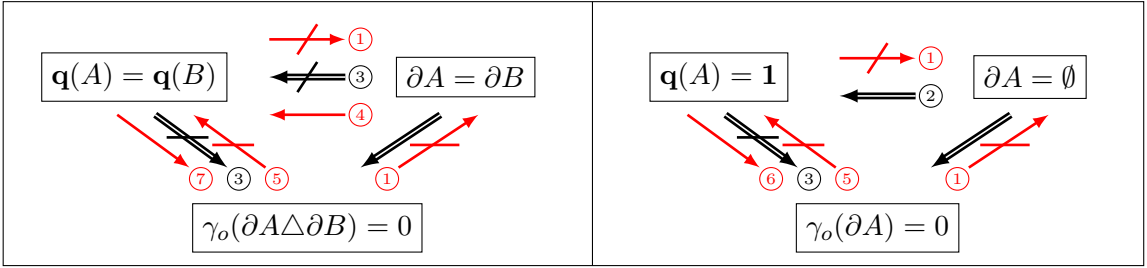
\begin{center}
\begin{tikzpicture}


\node[draw,right] at (0,0) {$\mathbf{q}(A) = \mathbf{q}(B)$};

\node[draw,left] at (6.5,0) {$\partial A = \partial B$};

\node[draw] at (3.5,-2) {$\gamma_o (\partial A \triangle \partial B)=0$};


\draw[->, >=latex, red, line width=1pt] (3,0.5) to (4,0.5);

\draw[red, line width=1pt] (3.37,0.3) -- (3.63,0.7);

\node[red] at (4.2,0.5) {{\tiny{1}}};

\draw [red] (4.2,0.5) circle [radius=0.15];

\draw[->, >=latex, red, line width=1pt] (4,-0.5) to (3,-0.5);

\node[red] at (4.2,-0.5) {{\tiny{4}}};

\draw [red] (4.2,-0.5) circle [radius=0.15];

\draw[double, ->, >=latex, black, line width=1pt] (4,0) to (3,0);

\draw[black, line width=1pt] (3.37,-0.2) -- (3.63,0.2);

\node at (4.2,0) {{\tiny{3}}};

\draw (4.2,0) circle [radius=0.15];


\draw[->, >=latex, red, line width=1pt] (2.95,-1.2) to (2,-0.5);

\draw[red, line width=1pt] (2.25,-0.9) -- (2.8,-0.9);

\node[red] at (3.1,-1.35) {{\tiny{5}}};

\draw [red] (3.1,-1.35) circle [radius=0.15];

\draw[double, ->, >=latex, black, line width=1pt] (1.5,-0.5) to (2.45,-1.2);

\draw[black, line width=1pt] (1.65,-0.8) -- (2.2,-0.8);

\node at (2.6,-1.35) {{\tiny{3}}};

\draw (2.6,-1.35) circle [radius=0.15];

\draw[->, >=latex, red, line width=1pt] (1,-0.5) to (1.95,-1.2);

\node[red] at (2.1,-1.35) {{\tiny{7}}};

\draw [red] (2.1,-1.35) circle [radius=0.15];


\draw[double, ->, >=latex, black, line width=1pt] (5.25,-0.5) to (4.2,-1.2);



\draw[<-, >=latex, red, line width=1pt] (5.75,-0.5) to (4.7,-1.2);

 \draw[red, line width=1pt] (4.85,-0.9) -- (5.5,-0.9);

\node[red] at (4.55,-1.35) {{\tiny{1}}};

\draw[red] (4.55,-1.35) circle [radius=0.15];


\draw [black] (-0.5,-2.5) rectangle (14.5,1);

\draw[-] (7.25,-2.5) -- (7.25,1);


\node[draw,right] at (8,0) {$\mathbf{q}(A) = \mathbf{1}$};

\node[draw,left] at (13.75,0) {$\partial A =\emptyset$};

\node[draw] at (11,-2) {$\gamma_o (\partial A)=0$};


\draw[->, >=latex, red, line width=1pt] (10.5,0.25) to (11.5,0.25);

\draw[red, line width=1pt] (10.87,0.05) -- (11.13,0.45);

\node[red] at (11.7,0.25) {{\tiny{1}}};

\draw [red] (11.7,0.25) circle [radius=0.15];

\draw[double, <-, >=latex, black, line width=1pt] (10.5,-0.25) to (11.5,-0.25);

\node at (11.7,-0.25) {{\tiny{2}}};

\draw (11.7,-0.25) circle [radius=0.15];


\draw[->, >=latex, red, line width=1pt] (10.45,-1.2) to (9.5,-0.5);

\draw[red, line width=1pt] (9.75,-0.9) -- (10.3,-0.9);

\node[red] at (10.6,-1.35) {{\tiny{5}}};

\draw [red] (10.6,-1.35) circle [radius=0.15];

\draw[double, ->, >=latex, black, line width=1pt] (9,-0.5) to (9.95,-1.2);

\draw[black, line width=1pt] (9.15,-0.8) -- (9.7,-0.8);

\node at (10.1,-1.35) {{\tiny{3}}};

\draw (10.1,-1.35) circle [radius=0.15];

\draw[->, >=latex, red, line width=1pt] (8.5,-0.5) to (9.45,-1.2);

\node[red] at (9.6,-1.35) {{\tiny{6}}};

\draw [red] (9.6,-1.35) circle [radius=0.15];


\draw[double, ->, >=latex, black, line width=1pt] (12.75,-0.5) to (11.7,-1.2);



\draw[<-, >=latex, red, line width=1pt] (13.25,-0.5) to (12.2,-1.2);

 \draw[red, line width=1pt] (12.35,-0.9) -- (13,-0.9);

\node[red] at (12.05,-1.35) {{\tiny{1}}};

\draw[red] (12.05,-1.35) circle [radius=0.15];

\end{tikzpicture}
\end{center}
\medskip

We start by observing that the two tables are related; indeed, by taking $B=\emptyset$ (or any finite set for that matter) on the left we obtain the corresponding one on the right. 
Unless stated otherwise, all the (counter)examples are on $\mathbb{T}_d$ for $d \ge 3$ and $o$ is the root; moreover $\lambda \in (1/d, \, 1/2\sqrt{d-1}]$ as in Section~\ref{sec:boundarymeasure} in such a way that there is a.s.~extinction on every finite subset but global survival with positive probability.

Here are the (numbered) comments on the above implications.
\begin{enumerate}
 \item[{\mycirc{1}}] Consider $A=\{x=\varphi(i)\colon i \in \mathbb{N}\}$ where $\varphi$ is a ray. It is known (see for instance \cite{BZ2020}) that $\gamma_o(\partial A)=0$ nevertheless $\partial A=\{\varphi\} \neq \emptyset$.
 \item[{\mycirc{2}}] Clearly $\partial A=\emptyset$ if and only if $A$ is finite, thus in this case there is a.s.~extinction (hence no persistence).
 \item[{\mycirc{3}}] It follows from  Proposition~\ref{prop:densenopersistence} and Remark~\ref{rem:densenopersistence}. Indeed 
 take $A_\emptyset$ and $\mathbb{T}_d$: they have the same boundary but different extinction probabilities. Consider now $\emptyset$ and $A_\emptyset$: they have the same extinction probabilities but $\gamma_o(\partial \emptyset \triangle \partial A_\emptyset)=1$.
 As for the remaining counterexample (on the right hand side of the table), just take $A_\emptyset$; $\gamma_o(\partial A_\emptyset)=1$ but $\mathbf{q}(A_\emptyset)=\mathbf{1}$.
 \item[{\mycirc{4}}] See Proposition~\ref{pro:connectedequalboundary}.
 \item[{\mycirc{5}}] See Examples~\ref{exmp:GWtree}~and~\ref{exmp:pruningtree} where there is persistence w.p.p.~and not just 
 survival. See also Example~\ref{exmp:cartesian}.
  \item[{\mycirc{6}}] See Theorem~\ref{prop:equivalent} which holds for generic trees (or Proposition~\ref{pro:survivalpositivemeasure} with a finite connected subset $B$). 
   \item[{\mycirc{7}}] See Proposition~\ref{pro:survivalpositivemeasure} which holds for generic trees. 
\end{enumerate}

\section{Proofs}\label{sect:proof}


\subsection{Proof of Theorem \ref{prop:equivalent} (i) and (ii)}

Theorem \ref{prop:equivalent} (i) is a consequence of Lemma \ref{lemma:rho=phi};
Theorem \ref{prop:equivalent} (ii) follows from Lemma  \ref{lemma:rhoVSMmin}; 
Theorem \ref{prop:equivalent} (iii) is a consequence of Proposition~\ref{pro:connectedequalboundary} and its proof is in Section~\ref{sec:proofPropconnectedequalboundary}.

We start with a preparatory result.
Suppose that  the RW governed by $P$
has period $\mathbf{d}\geq 1$, where
$\mathbf{d}= \min \{k \geq 1 \ : \ p^{(k)}(x,x)>0 , \text{ for all }x \in G\}$.
Then (if $U$ is connected), denote by $\mathbf{d}_U$ the period of the Markov chain induced by $U$; $\mathbf{d}_U$ is a multiple of $\mathbf{d}$. Recall that 
$\mathcal{E}_n(U)$ was defined in \eqref{eq:En}.

\begin{Lemma}\label{lemma:step1}
Let $U$ be a connected subgraph and
$x\in U$.
Then, $\lim_n \left ( \Prw_x \left [\mathcal{E}_n(U) \right ]\right )^{1/n} $ exists and it is equal to $  \frac{\rho_U}{\phi_U}$.
\end{Lemma}

\begin{proof}
Fix any $x\in U$, choose a sequence of vertices $y_1, \ldots, y_{\mathbf{d}_U-1} \in U$ such that for each $y_j$ ($j\in \{1, \ldots , \mathbf{d}_U-1\}$) one has
$
p_U^{(j)}(x,y_j)>0
$.
Since $U$ is connected and  $U\neq\{x_0\}$ with $p(x_0,x_0)=0$,
 this is always possible.
Let $\{R_n\}_n$ denote a nearest-neighbor RW induced by the kernel $P$ on $G$.
For all $k\ge0$
\begin{equation}\label{eq:aux-periodic}
\begin{split}
& p_U^{(k\mathbf{d}_U+ j)}(x,y_j) 
 = \Prw_x^P \left [ R_{k\mathbf{d}_U+ j}=y_j \cap \mathcal{E}_{k\mathbf{d}_U+ j}(U) \right ]  \\
 & =\Prw_x^P \left [ R_{k\mathbf{d}_U+ j}=y_j\mid \mathcal{E}_{k\mathbf{d}_U+ j}(U) \right ]\Prw_x^P \left [\mathcal{E}_{k\mathbf{d}_U+ j}(U) \right ]
= q_U^{(k\mathbf{d}_U+ j)}(x,y_j)\Prw_x^P \left [\mathcal{E}_{k\mathbf{d}_U+ j}(U) \right ].
\end{split}
\end{equation}
Moreover, 
\[
\begin{split}
\rho_U & = \limsup_n \left ( p_U^{(n  )}(x,x)\right )^{1/n} =\lim_{k} \left ( p_U^{(k\mathbf{d}_U )}(x,x) \right )^{1/(k\mathbf{d}_U)} 
\\
& 
=  \limsup_n \left ( p_U^{(n  )}(x,y_j)\right )^{1/n} = \lim_k \left ( p_U^{(k\mathbf{d}_U+ j)}(x,y_j)  \right )^{1/(k\mathbf{d}_U+j)},
\end{split}
\]
and a similar reasoning can be done for $\phi_U$.
Thus, by taking the $(k\mathbf{d}_U+j)$-th root in \eqref{eq:aux-periodic} and then the limit on $k\to \infty$, by \eqref{eq:phi_U} and \eqref{eq:defRhoU}, the above calculation gives the claim. 
\end{proof}
Recall the process $\{\mathbb{U}_n^x\}_n$  defined in Section \ref{sect:constr-propBZ}.

\begin{Lemma}\label{lemma:step2}
Let $U$ be a connected subgraph and
$x\in U$, then
\[
\lim_n \left[ \Ebrw [\sum_{w\in U}\mathbb{U}_{n}^x(w)] \right ]^{1/n} = \m \cdot \frac{\rho_U}{\phi_U}.
\]
\end{Lemma}
\begin{proof}
For all $w\in U$, by \eqref{eq:expectedBZn} and by the construction in Section \ref{sect:constr-propBZ} we have
\[
\Ebrw \left [ \mathbb{U}_1^x(w)  \right ] = \sum_{f_x \in S_U} f_x(w) \mu_x (f_x) = \m \cdot p_U(x,w).
\]
Hence, 
\[
\Ebrw \Big [\sum_{w\in U}\mathbb{U}_{1}^x(w) \Big ] = \sum_{w\in U} m_{x,w} =  \sum_{w\in U}  \m \cdot p_U(x,w) =\m \cdot \Prw_x \left ( \mathcal{E}_1(U)\right ),
\]
and repeating the same reasoning for all $n\geq 1$ we obtain
$
\Ebrw [\sum_{w\in U}\mathbb{U}_{n}^x(w)]=\m^n \cdot \Prw_x \bigl [\mathcal{E}_n(U) \bigr ]
$.
By taking the $n$-th root and limits on both sides,  Lemma \ref{lemma:step1} gives the claim.
\end{proof}

\begin{Lemma}\label{lemma:rho=phi}
Let $U$ be a connected subgraph and
$x\in U$.
If for all $\m\in (1, \rho_G^{-1}]$ there is positive probability of global persistence 
on $U$,
then $\rho_U=\phi_U$. 
Furthermore, if $\{\mathbb{U}_n^x\}_n$ is 
an $\mathcal{F}$-BRW, then also vice versa holds.
\end{Lemma}
\begin{proof}
Suppose that, for all $\m\in (1, \rho_G^{-1}]$,
$\{\B_n\}_n$
persists globally on $U$.
Then necessarily, by \cite[Theorem 4.1]{Zucca}, by Definition \ref{def:global-local} and Lemma \ref{lemma:step2} we have  
\[
\lim_n \Big (\Ebrw \Big [\sum_{w\in U}\mathbb{U}_{n}^x(w) \Big ]\Big )^{1/n} =\m \cdot \frac{\rho_U}{\phi_U} \geq 1.
\]
By assumption, the above must hold for every (fixed) $\m>1$, hence we can consider 
\[
\lim_{\m \to 1^+}\left [  \m \cdot \frac{\rho_U}{\phi_U}\right ] \geq 1,
\]
which can be satisfied only if $  \frac{\rho_U}{\phi_U}=1$.

Suppose now that $\rho_U=\phi_U$ which, by Lemma \ref{lemma:step1} is equivalent to  $\lim_n \left ( \Prw_x \bigl [\mathcal{E}_n (U)\bigr ]\right )^{1/n}=1$.
We are assuming that $\m>1$, then from Lemma \ref{lemma:step2} it follows
%
$\lim_n \left (\Ebrw [ \sum_{w\in U}\mathbb{U}_{n}^x(w)]  \right )^{1/n}  = \m>1$.
If $\{\mathbb{U}^x_n\}_n$ is 
an $\mathcal{F}$-BRW, by \cite[Theorem 4.3]{Zucca} one has global persistence of  $\{\mathbb{U}^x_n\}_n$
with positive probability.
Thus the claim.
\end{proof}

\begin{Remark}\label{rem:rho=phi}
 It is well known that, when
 the offspring distribution $\nu$ is equal at all sites, then
 $\rho_G$ is related to the local behavior of the BRW; local survival being equivalent to $\m>\rho_G^{-1}$, see for instance \cite[Proposition 4.33]{Bertacchi-Zucca-book}. 
 Given a fixed $\m$,
 when $\rho_U=\phi_U$, then
 two BRWs, the one governed by $P_U$ and the one governed by $Q_U$, either survive locally or they
 both die out. 
 It is natural to ask whether $\rho_U=\phi_U$ implies $\rho_U=\rho_G$. The answer is negative, even assuming positive probability of global persistence
 for all $\m\in(1,\rho_G^{-1}]$, as shown by the following example.
 
 Let $G$ be the union of the homogeneous tree $\mathbb{T}_d$ and an additional singleton $\{w_0\}$ ($w_0 \not \in \mathbb{T}_d$) where $w_0$ is connected only to the root $o \in \mathbb{T}_d$. Consider the following transition probability matrix
  \[
  p(x,y):=
  \begin{cases}
    p & \text{if } x=o,\, y=w_0\\
    1 & \text{if } x=w_0,\, y=o\\
   (1-p)/d & \text{if } x=o,\, y \sim o\\
   1/d & \text{if } x\not \in \{o, w_0\},\, y \sim x\\
  \end{cases}
 \]
where $p \in (0,1)$ and $x \sim y$ means that $x$ and $y$ are neighbors.
Let $U:=\mathbb{T}_d \subseteq G$. Denote by $Y_n(x)$ the ball of radius $n$ centered at $x$ in the natural metric of $G$. According to \cite[Theorem 5.1]{Zucca} the spectral radius of $P$ (resp.~$P_U$) is the limit, from below,  of the spectral radii of $P_{Y_n(x)}$ (resp.~$P_{Y_n(x) \cap U})$.  
Clearly $\rho_G \ge \sqrt{p}$ since $p^{(2n)}(o,o)\ge (p(o,w_0)p(w_0,o))^n \ge p^n$. Moreover $\rho_U \le 2\sqrt{d-1}/d$ since 
$P_U$ is bounded by the transition probability matrix of the simple RW on $\mathbb{T}_d$. 
 On the other hand, if $\{x_n\}$ is a sequence of vertices in $U$ such that $d(x_n,o) > n$ then 
 $\rho_U \ge \rho_{U \cap Y_n(x_n)}=\rho_{Y_n(x_n)}\uparrow 2\sqrt{d-1}/d$ (the last limit comes from \cite[Theorem 5.1]{Zucca} and the transitivity of the simple RW on $\mathbb{T}_d$. Whence 
 $\rho_U=2\sqrt{d-1}/d$.
 
It is not difficult to see that $Q_U$ is the transition matrix of the simple RW on $\mathbb{T}_d$, thus $\rho_U=\phi_U$. However $\rho_G>\rho_U$ if $p > 4(d-1)/d^2$.

Note that there is positive probability of global persistence in $U$ for all $\textbf{m} \in (1, \rho_G^{-1}]$ (this does not depend on the starting vertex); indeed, 
let us consider two BRWs: the one governed by $P$ with $p=0$ (BRW$_1$)
and the one governed by $P_U$ (BRW$_2$).
According to \cite[Theorem 2.4]{BZ2020}, the probability of never visiting $o$ (starting from $x \neq w_0$) is the same for both BRWs. Moreover, if $p=0$  then BRW$_1$ is essentially the BRW on $\mathbb{T}_d$ governed by the transition probability matrix of the simple RW and, by our choice of $\m$, $\mathbf{1}=\mathbf{q}^1(\{o, w_0\})>\mathbf{q}^1(G)$
(we denote by $\mathbf q^i$ the extinction probabilities of BRW$_i$).
By \cite[Theorem 2.4]{BZ2020} (take $A=\{o, w_0\}$ and $B=G$) for BRW$_2$ we have  $\mathbf{q}^2(\{o, w_0\})>\mathbf{q}^2(G)$, whence 
there is a positive probability of survival in $U$ starting from $x$ without visiting $o$ and this implies positive probability of persistence in $U$ (because there is no way to leave $U$ without visiting $o$).
\end{Remark}

Recall that $\gamma_o$ is the harmonic 
measure on $\M(P,1)$ induced by the RW $\{Y_n\}_n$ governed by $P$:
namely, $\gamma_o(\cdot)=\Prw^P_o(Y_\infty \in \cdot)$ where $Y_\infty:=\lim_n Y_n$.
Recall the definition of survival from Definition \ref{def:survival-U}. 

The next result gives a sufficient condition for survival (resp.~persistence) of a BRW in a neighborhood of a Borel subset of the boundary.

\begin{Lemma}\label{lemma:rhoVSMmin}
Consider a Borel set $A\subseteq \M(P,1)$ so that $\Prw^P_o(Y_\infty \in A)>0 $, and a fixed neighborhood $U$ of $A$.
Then for all $\m>1$,
$
\{\mathbf{B}_n\}_n \text{ survives in }U
$ starting from every $x \in G$.
Moreover, if $U$ is connected, then
$
\{\mathbf{B}_n\}_n \text{ persists in }U
$ starting from every $x \in U$.
\end{Lemma}
\begin{proof}
According to \cite[5.34 Fact $\#$4]{Woess-denumerable}, if we fix an infinite genealogy tree of the BRW (that is, the number of children of each descendant, disregarding their position) and one infinite line of descent, then the (random) sequence of positions of the particles along this line 
is a RW with transition matrix $P$. Then, by definition of $\gamma_o$, the probability of convergence to $A$, conditioned on this fixed infinite genealogy tree is positive. Since there is a positive probability of global survival, that is, a positive probability of an infinite genealogy tree, then there is a positive probability that a line of descent of the BRW converges to $A$. Then the line of descent 
eventually stays  in $U$ with positive probability and this implies survival in $U$.
In order to have persistence, we must prove that, with positive probability, there is at least one line of descent which stays in $U$ from the start.
To this aim, suppose that $U$ is connected and fix $x \in U$.
Consider all the above lines of descendants and the (countable) set $V_U$ of vertices where they last enter $U$; each one of these vertices can be reached, with positive probability, by a path starting from $x$ and never exiting $U$ (since $U$ is connected). There is at least one vertex in $V_U$ such that, with positive probability, the RW does not exit $U$ and this vertex can be reached with positive probability from $x$ without leaving $U$.
This implies positive probability of persistence from any starting vertex $x \in U$.
\end{proof}

\subsection{Proof of Theorem \ref{thm:local-surv} and well-posedness of \eqref{eq:defm1} }\label{sect:proof-thm}

We start with an easy auxiliary result.
\begin{Lemma}\label{lemma:step3}
Fix a connected subgraph $U$, and
let $x\in U$ be a fixed vertex.
If the induced process $\{\mathbb{U}^x_n\}_n$ is 
an $\mathcal{F}$-BRW, then
$
\m_1= \frac{\phi_U}{\rho_U}
$.
Otherwise, $\m_1\geq \frac{\phi_U}{\rho_U}$.
\end{Lemma}
\begin{proof}
By Lemmas \ref{lemma:step1} and \ref{lemma:step2} we get that for every starting vertex $x\in U$
\[
\lim_n \Big (\Ebrw \Big [ \sum_{w\in U}\mathbb{U}_{n}^x(w) \Big ] \Big )^{1/n} = \m \cdot \liminf_n \sqrt[n]{\Prw_x \bigl [\mathcal{E}_n(U) \bigr ]}  = \m \cdot \frac{\rho_U}{\phi_U}.
\]
Thus, the condition $\m \cdot \frac{\rho_U}{\phi_U}<1$ ensures (by \cite[Theorem 4.1]{Zucca}) a.s.~global extinction of $\{\mathbb{U}_n\}_n$.
In other words, $\{\B_n\}_n $ started at $x$ will exit $U$ almost surely.
Hence $\m_1\geq \frac{\phi_U}{\rho_U}$.
Note that this fact does not imply that for $\m<\m_1$ the set $U$ is only visited finitely many times by $\{\B_n\}_n$: there might be trajectories of $\{\B_n\}_n$ that enter and exit $U$ infinitely many times.
If, in addition, the auxiliary process $\{\mathbb{U}_n^x\}_n$ is 
an $\mathcal{F}$-BRW, by \cite[Theorem 4.3]{Zucca}, condition 
$
\m \cdot \frac{\rho_U}{\phi_U}>1
$
implies global persistence of $\{\mathbb{U}_n^x\}_n$
with positive probability, 
hence $\m_1\leq \frac{\phi_U}{\rho_U}$.
\end{proof}

We are now able to conclude the proof of the theorem.

\begin{proof}[Proof of Theorem \ref{thm:local-surv}]
Lemma \ref{lemma:step3} already proves the first inequality; in order to prove the second one we employ \cite[Theorem 4.1]{Zucca}.
More precisely, we shall prove that the BRW started at $x$ and killed upon exiting $U$ persists locally in $U$ 
if and only if $\m>\rho_U^{-1}$.

We proceed as follows.
For any fixed vertex $w\in U$, consider the auxiliary BRW $ \{\mathbb{U}_n^x( w) \}_n$. 
(Note that we are \emph{not} summing over $w$ here).
In particular we shall be interested in $ \{\mathbb{U}_n^x(x) \}_n$, with initial condition $\mathbb{U}_0^x(w)=\delta_x(w)$.
That is, we shall investigate the set of particles of the original BRW on $G$ such that there are descendants of the initial particle at $x$, that visit $x$ at some time $n\geq 1$, and whose family line has never exited $U$.
Now, by \cite[Theorem 4.1]{Zucca} local persistence occurs 
with positive probability 
if and only if
\[
\limsup_n \left ( \Ebrw [\mathbb{U}_n^x(x)]\right )^{1/n}>1.
\]
By construction, we have that $\Ebrw [\mathbb{U}^x_n(x)] = \m^n \cdot p_U^{(n)}(x,x)$,
then the above corresponds to 
\[
\limsup_n \left ( \Ebrw  [\mathbb{U}_n^x( x)]\right )^{1/n} = \limsup_n \left ( \m^n p_U^{(n)}(x,x)\right )^{1/n} = \m \, \rho_U>1,
\]
which implies the statement.
\end{proof}

At this point, the proof of the subsequent corollary is an easy fact.

\begin{proof}[Proof of Corollary \ref{cor:m1+}]
The first part is a direct consequence of Lemma \ref{lemma:step3}.
Furthermore, since local persistence implies global persistence, the second part follows directly from Lemma \ref{lemma:step3} and Theorem \ref{thm:local-surv}.
\end{proof}

We observe that, in the above result, the possibility $\phi_U=1$ is completely ruled out if we are seeking global but not local persistence in $U$.
%
%
Now we show under which conditions the quantity $\m_1$ defined in \eqref{eq:defm1} is actually a threshold.

\begin{Lemma}\label{pro:m1}
 Let $X$ be a generic set and $P'$ a substochastic matrix on $X$.
Given  an offspring distribution $\nu$ with expected value $\m$,  consider the modification of the usual BRW where each offspring of a particle at $x$ is placed at $y$ (resp.~killed) with probability $p'(x,y)$ (resp.~$1-\sum_{y \in X} p'(x,y)$). Then there exists $\m_{P'} \in [1, +\infty)$ such that $\m >\m_{P'}$ implies global survival w.p.p., while $\m < \m_{P'}$ implies a.s.~global extinction.
In particular, for the BRW induced by $U \subseteq X$, $\m >\m_1(U)$ implies 
persistence w.p.p., while $\m <\m_1(U)$ implies a null probability of persistence.
\end{Lemma}

\begin{proof} 
 Consider the BRW on $G \times \{0,1\}$ with offspring distribution $\nu$ and stochastic matrix $P$ defined as 
 \[
  p((x,i)\,,(y,j)):=
  \begin{cases}
   p'(x,y) & \text{if } i=j=0\\
   1-\sum_{w \in G} p'(x,w) & \text{if } x=y, \, i=0, \, j=1\\
   1 & \text{if } x=y, \, i=j=1\\
   0 & \text{otherwise.}
  \end{cases}
 \]
 It is clear that global extinction of 
 $\{\B_n\}_n$ 
 is equivalent to extinction of the new BRW in the set $G \times\{0\}$ or, by using the terminology of \cite{Hutchcroft2020Branching}, to $\nu$-transience of the new BRW in $G \times\{0\}$.
 The result follows from
 \cite[Corollary 1.2]{Hutchcroft2020Branching}.
 
 Consider now the BRW induced by $U \subseteq G$. It is clear that $P'_U$ is a substochastic matrix, hence the above arguments apply to the induced BRW as well.
\end{proof}

\subsection{Proof of Theorem \ref{thm:general-bz-bdary} }\label{sect:proof-bz}

The auxiliary process $\{\mathbb{U}_n^x\}_n $ is, in general, a BRW with infinitely many types.
BRWs with countably many types are yet not well understood, and only recently Hautphenne et al., and Bertacchi et al.\ (see e.g.\  \cite{HautphenneEtAl} and  \cite{BertacchiZuccaAndFriends} and references therein) managed to shed some light into such a difficult topic.
However, for 
$\mathcal{F}$-BRWs, one can prove several useful results.
A particular case that we are interested in, is when $U$ is 
such that $\sum_{y \in U} p(x,y)=\bz$ does not depend on $x \in U$, which means that, the global behavior of the BRW induced by $U$ is the same as  the global behavior of a BRW with just one type.
Then, for all vertices $x\in U$ the following quantity
$
\widetilde{m}_{x}:=\Ebrw \left [\sum_{w \in U}\mathbb{U}_{1}^x(w) \right ] =  \m \sum_{y\in U }p_U(x,y)
$
does not depend on $x$, that is $\widetilde{m}_{x}=\m \bz$ for all $x$. As an example, consider the case when $P_U$ is transitive.
We are now able to prove Theorem \ref{thm:general-bz-bdary}.

\begin{proof}[Proof of Theorem \ref{thm:general-bz-bdary}]
Proof of (i).
 By hypothesis we have that $\sum_{y \in U} p(x,y)=\bz$ and 
 $
\Prw_x(\mathcal{E}_j) = \bz^j
$ 
for all $x \in U$ and for all $j\geq 1$.
As a consequence, we have that, for all $x,y\in U$
\[
\begin{split}
G_{P_U}(x, y \mid \bz) 
& =  \sum_{j\geq 0} p_U^{(j)}(x, y)\frac{1}{\bz^j} =  \sum_{j\geq 0} \frac{p_U^{(j)}(x, y)}{\Prw_x(\mathcal{E}_j)}
=  \sum_{j\geq 0} q_U^{(j)}(x, y)= G_{Q_U} (x, y\mid 1).
\end{split}
\]
Hence, letting $o_U$ be a reference vertex in $U$,
\begin{equation}\label{eq:equal-kernels}
K_{P_U}(x,y \mid \bz) = \frac{G_{P_U}(x, y \mid \bz) }{G_{P_U}(o_U, y \mid \bz) } = \frac{G_{Q_U} (x, y\mid 1)}{G_{Q_U} (o_U, y\mid 1)} = K_{Q_U}(x,y\mid 1).
\end{equation}
Now, consider the functions $w_x(y):=K_{P_U}(x,y \mid \bz)$. Each of them can be extended to a unique continuous function, $\bar w_x$ on 
$\M(P_U, \bz)$.
Namely, for each sequence $\{y_n\}_n\subseteq U$ such that $y_n\to \xi \in  \M(P_U, \bz)$, 
we define
$\bar w_x(\xi):=\lim_n K_{P_U}(x,y_n \mid \bz)
$.
Then, by \eqref{eq:equal-kernels} we also have that
$
w_x(\xi)=\lim_n  K_{Q_U}(x,y_n\mid 1)
$.
In conclusion, since the above is true for all $x\in U$, then (by uniqueness of the Martin compactification), one has
\[
\M(P_U, \bz) 
\subseteq \M(Q_U, 1).
\]
The inverse inclusion can be proven in the same way.
Notice that we are not allowed to include the element $\eth_U$ in $\M(Q_U, 1)$, because a RW governed by $Q_U$ does not exit $U$, hence it will not get ``killed''.

Proof of (ii).
Here we use the machinery developed in Section \ref{sect:constr-propBZ} together with Theorem \ref{thm:notes}.
Since $P_U$ is transitive then $\sum_{y \in U} p(x,y)$ is constant on $U$ and $\{\mathbb{U}_n\}_n$ is 
an $\mathcal{F}$-BRW with only one type.
Moreover, $\nu$ satisfies a $L \log L$ condition by assumption, 
then the same condition is satisfied by the law of the variables $\sum_{w \in U}\mathbb{U}_{1}^x(w)$ as well.
Indeed, let $\nu_1$ be the common law of the variables $\sum_{w \in U}\mathbb{U}_{1}^x(w)$;
in the induced process there cannot be more offspring than in the original one, whence it is clear that $\nu_1$ is stochastically dominated by $\nu$. Since $f(x):=x \log(x)$ is a nondecreasing 
map on $\mathbb{N}$ then $\mathbb{E}_{\nu_1}[f] \le \mathbb{E}_{\nu}[f] <+\infty$.
%
%
Theorem \ref{thm:notes} applied to this context yields the result.
\end{proof}

\subsection{Proof of Propositions~\ref{prop:fbrw} and~\ref{prop:monotonicity}}\label{sect:proofFBRW}

Proposition~\ref{prop:fbrw} is a direct consequence of the following result, which provides sufficient conditions on $U$ so that the BRW induced by $U$ is projected on a BRW on a set $Y$; by taking a finite $Y$ we have an $\mathcal{F}$-BRW. Given a family of offspring distributions $\{\nu_x\}_{x \in G}$, we consider the BRW induced by $U$ which is associated with the probability measures $\mu^{(U)}_x$ defined in \eqref{eq:muxinduced}, with $\nu_x$ (dependent on $x \in U$) instead of $\nu$.
Henceforth, given a function $g:U \mapsto Y$, we say that a family of offspring distributions $\{\nu_x\}_{x\in G}$ is $g$-invariant if $g(x)=g(y)$ implies 
$\nu_x =\nu_y$
for all $x,y \in U$. 

\begin{Proposition}\label{prop:fbrwgen}
The following are equivalent.
\begin{enumerate}[(a)]
	\item
	There exists a surjective map $g:U\to Y$
	such that, for every $g$-invariant family of offspring distributions,
	the BRW induced by $U$  is projected on a BRW on $Y$ with projection $g$.
	\item
	There exists a surjective map $g:U\to Y$
	and a
	$g$-invariant family of offspring distributions $\{\nu_x\}_{x \in G}$ satisfying $\nu_x(0)<1$ for all $x \in U$, 
	such that 
	the BRW induced by $U$ is projected on a BRW on $Y$ with projection $g$.
	\item
	There exists a surjective  map $g:U\to Y$ 
	such that 
	the quantity $\sum_{w\colon g(w)=y}p_U(x,w)$ ($x \in U$) only depends on $g(x)$ and $y$.
\end{enumerate}
Finally, suppose that $U=G$; if the  BRW on $G$ is projected on a BRW on $Y$ with projection $g$ then the family $\{\nu_x\}_{x \in G}$ is $g$-invariant.

Suppose, in addition, that $\nu_x(0)<1$ for all $x \in G$; then the BRW on $G$ is projected on a BRW on $Y$ with projection $g$ if and only if (c) holds and the family $\{\nu_x\}_{x \in G}$ is $g$-invariant.
\end{Proposition}
\begin{proof}
	To avoid a cumbersome notation, in this proof we write $\mu_x$ instead of $\mu^{(U)}_x$.
	By Definition~\ref{def:F-BRW}, the BRW induced by $U$ is projected on the BRW $(Y, \{\eta_y\}_{y\in Y})$
	(where 
	$\{\eta_y\}_{y\in Y}$ is a collection of probability measures on $S_Y$) if there exists a surjective map $g:U\to Y$ , such that for all $h\in S_Y$,
	\begin{equation}\label{eq:mueta}
		\eta_{g(x)}(h)=\mu_x(\pi_g^{-1}(h)),
	\end{equation}
	where $\pi_g:S_U\to S_Y$, $\pi_g(f)(y)=\sum_{z\in g^{-1}(y)} f(z)$.
	(a) $\Rightarrow$ (b). It is straightforward, once we prove that there exists at least one $g$-invariant family  $\{\nu_x\}_{x \in G}$ such that $\nu_x(0)<1$ for all $x \in U$. Take $\nu_x=\nu$ for all $x\in G$, where $\nu$ is a probability measure on $\mathbb N$ with $\nu(0)<1$.\\
	(b) $\Rightarrow$ (c). Let $g$ and  $\{\nu_x\}_x$ be as in (b).
	By hypothesis 
	\eqref{eq:mueta} holds for a suitable offspring family $\{\eta_y\}_{y\in Y}$.
	Let us compute $\mu_x(\pi_g^{-1}(h))$ for a generic $h \in S_Y$; observe that if $\pi_g(f)=h$ then $|f|=|h|$ (where $|f|=\sum_{x \in G} f(x)$). We have
	\[
	\begin{split}
		\eta_{g(x)}(h)&=\mu_x(\pi_g^{-1}(h))=
		\sum_{f \in \pi_g^{-1}(h)} \mu_x(f)\\
		&=
		\sum_{f \in \pi_g^{-1}(h)}  \sum_{n=|f|}^\infty \nu_x(n) \Big (
		1- \sum_{w \in U} p_U(x,w)) \Big )^{n-|f|}\binom{n}{|f|}\frac{|f|!}{\prod_w f(w)!}\prod_{w\in U}p_U(x,w)^{f(w)}\\
		&=
		\sum_{n=|h|}^\infty  \nu_x(n) \Big (
		1- \sum_{w \in U} p_U(x,w)) \Big )^{n-|h|}\binom{n}{|h|}\sum_{f \in \pi_g^{-1}(h)} \frac{|h|!}{\prod_{w \in U} f(w)!}\prod_{w \in U}p_U(x,w)^{f(w)}.
	\end{split}
	\]
	Observe that if $f\in \pi^{-1}_g(h)$ and $f_y:=f|_{g^{-1}(y)}$ then $|f_y|=h(y)$ for all $y \in Y$.
	Let us consider the last sum in the above equation
	\[
	\begin{split}
		\sum_{f \in \pi_g^{-1}(h)}& \frac{|h|!}{\prod_{w \in U} f(w)!}\prod_{w \in U} p_U(x,w)^{f(w)}=
		\sum_{f \in \pi_g^{-1}(h)} \frac{|h|!}{\prod_{w \in U} f(w)!}\prod_{y \in Y} \prod_{w \colon g(w)=y}p_U(x,w)^{f(w)}\\
		&=
		\sum_{f \in \pi_g^{-1}(h)} \frac{|h|!}{\prod_{y \in Y} h(y)!}\prod_{y \in Y} \Big [\frac{h(y)!}{\prod_{w \colon g(w)=y} f(w)!} \prod_{w \colon g(w)=y}p_U(x,w)^{f(w)}
		\Big ]\\
		&=
		\frac{|h|!}{\prod_{y \in Y} h(y)!}\sum_{\stackrel{\{f_y\}_{y \in Y} \colon}{f_y \in S_{g^{-1}(y)},|f_y|=|h(y)|, \forall y \in Y}} \prod_{y \in Y} \Big [\frac{|f_y|!}{\prod_{w \colon g(w)=y} f_y(w)!} \prod_{w \colon g(w)=y}p_U(x,w)^{f_y(w)}
		\Big ]\\
		&=
		\frac{|h|!}{\prod_{y \in Y} h(y)!}
		\prod_{y \in Y} 
		\Big (\sum_{w \colon g(w)=y} 
		p_U(x,w)
		\Big )^{h(y)}.
	\end{split}
	\]
	Whence
	\begin{equation}\label{eq:2ndformmeasure}
		\mu_x(\pi_g^{-1}(h))=
		\sum_{n=|h|}^\infty  \nu_x(n) 
		\frac{n!}{(n-|h|)!\prod_{y \in Y} h(y)!}
		\Big (
		1- \sum_{w \in U} p_U(x,w)) \Big )^{n-|h|}
		\prod_{y \in Y} 
		\Big (\sum_{w \colon g(w)=y} 
		p_U(x,w)
		\Big )^{h(y)}.
	\end{equation}
	Since $\eta_{g(x)}(h)=\mu_x(\pi_g^{-1}(h))$ then the R.H.S. of \eqref{eq:2ndformmeasure} only depends on $g(x)$ for each fixed $h$.
	If we consider $h:=\mathbf{0}$, then $\pi_g^{-1}(h)=\mathbf{0}$ and
	\[
	\mu_x(\pi_g^{-1}(h))=\mu_x(\mathbf{0})=\sum_{n=0}^\infty \nu_x(n)\left(1-\sum_{w\in U}p_U(x,w)\right)^{n}=\eta_{g(x)}(0),
	\]
	which depends on $x$ only through $g(x)$. Since $\nu_x(0)<1$ then $z \mapsto \sum_{n=0}^\infty \nu_x(n) z^n$ is strictly increasing; moreover $\nu_x$ only depends on $g(x)$, thus $\sum_{w\in U}p_U(x,w)$ only depends on $g(x)$.
	If we take now $h=k\delta_y$, where $k>0$ is such that $\nu(k)>0$, then
	$\eta_{g(x)}(k \delta_y)$ again must depend only on $g(x)$ and coincide with $\mu_x(\pi_g^{-1}(k \delta_y))$.
	By definition of the map $\pi_g$,
	\[
	\pi_g^{-1}(k\delta_y)=\{f\in S_U\colon |f|=k, \mathrm{supp}(f) \subseteq g^{-1}(y)\}.
	\]
	\[\begin{split}
		\mu_x(\pi_g^{-1}(k\delta_y))&=
		\sum_{n=k}^\infty \nu_x(n)
		\binom{n}{k} \left(1-\sum_{z\in U}p_U(x,z)\right)^{n-k}
		\left (\sum_{w\colon g(w)=y}p_U(x,w) \right )^k.\\
	\end{split}
	\]
	This quantity must depend on $x$ only through $g(x)$ and, since $\sum_{z\in U}p_U(x,z)$ only depends on $g(x)$, then 
	$\sum_{n=k}^\infty \nu_x(n) \binom{n}{k} \left(1-\sum_{z\in U}p_U(x,z)\right)^{n-k}>0$
	only depends on $g(x)$.
	In turn, this implies that 
	$\sum_{w\colon g(w)=y}p_U(x,w)$ only depends on $g(x)$ and $y$, which is the claim.
	\\
	(c) $\Rightarrow$ (a).
	Let  $\pi_g:S_U\to S_Y$ be as in Definition \ref{def:F-BRW},  consider a fixed family of $g$-invariant offspring distributions $\{\nu_x\}_{x \in G}$ and the related BRW where $\mu_x$ is defined as in equation~\eqref{eq:muxinduced} with $\nu_x$ instead of $\nu$.
	For any given $h \in S_Y$, by  equation~\eqref{eq:2ndformmeasure}, we have that
	\[
	\mu_x(\pi_g^{-1}(h))=
	\sum_{n=|h|}^\infty  \nu_x(n) \frac{n!}{(n-|h|)!\prod_{y \in Y} h(y)!}\Big (
	1- \sum_{w \in U} p_U(x,w)) \Big )^{n-|h|}
	\prod_{y \in Y} 
	\Big (\sum_{w \colon g(w)=y} 
	p_U(x,w)
	\Big )^{h(y)} 
	\]
	and from (c) we have that the R.H.S.~only depends on $g(x)$.
	Whence $\eta_{g(x)}(h):=\mu_x(\pi_g^{-1}(h))$ is a well-posed definition of a family of probability measures $\{\eta_y\}_{y \in Y}$ on $S_Y$. Thus, the BRW on $U$ is clearly projected on $(Y,\{\eta_y\}_{y \in Y})$.
	
	Finally,
	suppose that the BRW on $G$ is projected on $(Y, \{\eta_y\}_{y \in Y})$ with projection $g$ (here $U=G$); thus $\mu_x=\mu_x^{(G)}$ and, by using equation~\eqref{eq:defmu} (with $\nu_x$ instead of $\nu$) for all $n \in \mathbb{N}$, it is easy to see that
	\[
	\nu_x(n)=	\mu_x(f \in S_G \colon |f|=n)=\mu_x \big (\pi_g^{-1}(h \in S_Y\colon |h|=n)\big )=\eta_{g(x)}\big (\pi_g^{-1}(h \in S_Y\colon |h|=n)\big )
	\]
	which depends only on $g(x)$. Hence, $\{\nu_x\}_{x \in G}$ is $g$-invariant (we do not need here the additional hypothesis $\nu_x(0)<1$ for all $x \in G$).
	The claim follows from the equivalence between (a), (b) and (c).
\end{proof}

Now we can prove Propositions \ref{prop:fbrw} and \ref{prop:monotonicity}.

\begin{proof}[Proof of Proposition~\ref{prop:fbrw}]
 When $\nu_x$ does not depend on $x \in U$ then, 
the family $\{\nu_x\}_{x \in G}$ is trivially $g$-invariant, for every function $g$ defined on $U$. Thus, this condition is unnecessary and can be removed from the statement of Proposition~\ref{prop:fbrwgen} obtaining the statement of Proposition~\ref{prop:fbrw} when $Y$ is finite. 
\end{proof}

\begin{proof}[Proof of Proposition~\ref{prop:monotonicity}]
According to \cite[Theorem 2.4]{BZ_genfunapproach}, positive probability of global survival for an $\mathcal F$-BRW on $U$ is equivalent to 
$\liminf_{n\to\infty}\sqrt[n]{\sum_{y\in U}m^{(n)}_{x,y}}>1$. Since here $m_{x,y}=\mathbf{m}\cdot p_U(x,y)$, we have the claim.
\end{proof}

\subsection{Proof of Proposition~\ref{prop:densenopersistence}}\label{sect:proofdensenopersistence}

For convenience of the reader, we recall the following result, which will be useful to prove positive probability of survival in subsets
(see \cite[Theorem 3.3]{BZ_stronglocal} and  \cite[Theorem 4.1]{BertacchiZuccaAndFriends} for more details).
Here $\mathbf{q}_0(x,B)$ denotes the probability that the BRW starting with one particle in $x$ never visits $B$. 

\begin{Theorem}\label{th:1}
For any BRW on a set $\mathcal{X}$  
and $A,B \subseteq \mathcal{X}$, the following statements are equivalent:
\begin{enumerate}[(i)]
\item there exists $x \in \mathcal{X}$ such that $\mathbf{q}(x,A) < \mathbf{q}(x,B)$
\item there exists $x \in \mathcal{X}$ such that $\mathbf{q}(x,A) < \mathbf{q}_{0}(x,B)$
\item there exists $x \in \mathcal{X}$ such that,
starting from $x$
there is a positive probability of survival in $A$ without ever visiting $B$
\item there exists $x \in \mathcal{X}$ such that, starting from $x$ there is a positive probability of survival in $A$ and extinction in $B$
starting from $x$
\item 
\[
\inf_{x \in \mathcal{X}\colon \mathbf{q}(x,A)<1} \frac{1-\mathbf{q}(x,B)}{1-\mathbf{q}(x,A)}=0.
\]
\item $\mathbf{q}(A \cup B) < \mathbf{q}(B)$.
\end{enumerate}
In particular if any of the above holds, then $\mathbf{q}(A)< \mathbf{1}$ and $\mathbf{q}(\mathcal{X}) < \mathbf{q}(B)$.

\noindent Finally, if we replace $A$ by $A \setminus B$ in $(i)-(vi)$ we obtain a set of conditions equivalent to the above ones.
\end{Theorem}

The starting points $x$ in Theorem \ref{th:1} are related. 
The vertices $x$ satisfying (i) and (iv) are the same. The vertices $x$ satisfying (ii) and (iii) are the same. 
If $x$ satisfies, say, (ii) then it satisfies 
(i).
On the other hand, 
 it follows from the proof that if $x$ satisfies, say (i) then every
$x^\prime$ satisfying (ii) must be among the vertices reachable by the progeny of a particle living at $x$.

In order to prove Proposition~\ref{prop:densenopersistence} we need a lemma regarding the behavior of $\mathbf{q}_0(y,\{x\})$ on $\mathbb{T}_d$. 

\begin{Lemma}\label{lem:densenopersistence}
%
Consider the edge-breeding BRW on $\mathbb{T}_d$ with $1/d < \lambda \le 1/2\sqrt{d-1}$.
For any fixed $y \in \mathbb{T}_d$ and $x_1, x_2 \in \mathbb{T}_d $ such that $d(y,x_1)< d(y,x_2)$ we have $\mathbf{q}_0(y, \{x_1\})<\mathbf{q}_0(y,\{x_2\})$. Moreover, there exists $c >0$ and $\delta <1
 $ such that $1-\mathbf{q}_0(y,\{x\}) \le c \delta^{d(y,x)}$.
\end{Lemma}

\begin{proof}
 Let us note that, by symmetry, $\mathbf{q}_0(y,\{x\})=\mathbf{q}_0(x,\{y\})$ and this only depends on $d(x,y)$. Moreover, we know that if $y \neq o$ (where $o$ is the root of the tree $\mathbb{T}_d$) then 
 ${1} > \mathbf{q}(x,\mathbf{T}_y)>
 \mathbf{q}(x,\mathbb{T}_d)$ for all $x \in \mathbb{T}_d$. 
 Again by symmetry and using the results of \cite{BZ2020}, if we define 
 \[
  r(x,y):=
  \begin{cases}
   d(x,y) & \text{if } x \not \in \mathbf{T}_y \\
   -d(x,y) & \text{if } x \in \mathbf{T}_y\\
  \end{cases}
   \]
then $\mathbf{q}(x, \mathbf{T}_y)$ is strictly increasing with respect to $r(x,y)$. Moreover, by \cite[Corollary 4.2]{BertacchiZuccaAndFriends} we know that $\sup_{x \in \mathbb{T}} \mathbf{q}(x, \mathbf{T}_y)=1$. As a consequence, by the above monotonicity,
\begin{equation}\label{eq:limitmoyal}
\lim_{r(x,y) \to +\infty} \mathbf{q}(x, \mathbf{T}_y)=
 \lim_{d(x,y) \to +\infty, \, x \not \in \mathbf{T}_y} \mathbf{q}(x, \mathbf{T}_y)  =1.
 \end{equation}
 By the symmetry of the tree, it
 is enough to prove that
 $\lim_{x \to \partial \mathbb{T}_d} \mathbf{q}_0(o,\{x\})=1$ or, equivalently,
 $\lim_{x \to \partial \mathbb{T}_d} \mathbf{q}_0(x,\{o\})=1$. 
To this aim, let us fix a neighbor $x_1$ of $o$ and the subtree $\overline{ \mathbf{T}}:=\mathbb{T}_d \setminus \mathbf{T}_{x_1}$. In order to survive in $\overline{ \mathbf{T}}$ for a BRW starting with one particle in $x \not \in \overline{\mathbf{T}}$ it is necessary to visit $o$,
 therefore
 \[
1-\mathbf{q}(x, \overline{ \mathbf{T}}) \ge (1-\mathbf{q}_0(x,\{o\})) (1-\mathbf{q}(o, \overline{\mathbf{T}})).
 \]
 From the above inequality and equation~\eqref{eq:limitmoyal}, since $\mathbf{q}(o,\overline{\mathbf{T}})<1$ we have that 
  $\lim_{x \to \partial \mathbb{T}_d} \mathbf{q}_0(x,\{o\})=1$.
  
Let us prove 
monotonicity. As before, by symmetry, it is enough to prove the result when $y=o$.  By (2.2) in \cite[Section 3.2]{BZ2020} we have that 
${\bf G}(\mathbf{q}_0(\{o\})| x)=\mathbf{q}_0(x)$ for all $x \neq o$, where $\bf{G}$ is the generating function of the BRW as defined by equation~(3.3) in \cite[Section 3.2]{Zucca}. In this case ${\bf G}(\mathbf{v}|x)=1/(1+M(1-\mathbf{v})(x))$, where $M$ is the 1st moment matrix of the BRW
(see equation (3.4) in \cite[Section 3.2]{Zucca}). Whence, if $a_n:=1-\mathbf{q}_0(o,\{x\})=1-\mathbf{q}_0(x,\{o\})$ (where $d(o,x)=n$) then
  \begin{equation}\label{eq:genfuncttree}
   a_n=\frac{\lambda d \big (a_{n-1}/d+(d-1) a_{n+1}/d \big )}{1+\lambda d \big (a_{n-1}/d+(d-1) a_{n+1}/d \big )}.
  \end{equation}
  We prove by induction on $n$ that $a_n>a_{n+1}$. 
  It is clear that if $x \neq o$ then $\mathbf{q}_0(o,\{x\})=\mathbf{q}_0(o,\mathbf{T}_x)>0=\mathbf{q}_0(o,\{o\})$ (by using 
Theorem~\ref{th:1}), whence $a_0>a_1$.
  Suppose that $a_{n-1}>a_n$, then
  according to the Maximum Principle \cite[Proposition 2.1]{BZ_stronglocal},
  %
  %
  if ${\bf G}(\mathbf{v}|x) \ge \mathbf{v}(x)$ then either $\mathbf{v}(y)=\mathbf{v}(x)$ for all neighbors $y$ of $x$ or there exists a neighbor $y$ of $x$ such that $\mathbf{v}(y) >\mathbf{v}(x)$. 
  Now, since $1-a_{n-1}<1-a_n$ then 
$1-a_{n+1} > 1-a_n$, that is, $a_n > a_{n+1}$.
If we solve \eqref{eq:genfuncttree} for $a_{n+1}$ we have
\begin{equation}\label{eq:genfuncttree2}
 a_{n+1}=\frac{a_n}{\lambda(d-1)(1-a_n)}-\frac{a_{n-1}}{d-1}. 
\end{equation}
Since $a_n <a_{n-1}$ from the previous equation we have
\[
 a_{n+1}=\frac{a_n}{\lambda(d-1)(1-a_n)}-\frac{a_{n}}{d-1}=a_n \frac{d-(1+\Delta)(1-a_n)}{(1+\Delta)(d-1)(1-a_n)}.
\]
where $\Delta:= d \lambda -1$ (recall that $1/d < \lambda \le 1/(2 \sqrt{d-1})$).
Since $a_n \to 0$ we have 
$(d-(1+\Delta)(1-a_n))/((1+\Delta)(d-1)(1-a_n)) \to (1-\Delta/(d-1))/(1+\Delta)) <1$ eventually as $n \to +\infty$.
Whence there exists $\varepsilon>0$ such
that $ (d-(1+\Delta)(1-a_n))/((1+\Delta)(d-1)(1-a_n))\le 1-\varepsilon$ eventually as $n \to +\infty$; thus, 
by using \eqref{eq:genfuncttree2},
$a_{n+1} \le a_n (1-\varepsilon)$ eventually as $n \to +\infty$. This implies that the sequence
$a_n/(1-\varepsilon)^n$  is eventually non-increasing as $n \to +\infty$, thus there exists a finite constant $c \ge a_n/(1-\varepsilon)^n$ for all $n \in \mathbb{N}$. The claim follows by taking $\delta:=1/(1-\varepsilon)$.

\end{proof}

We can now prove Proposition~\ref{prop:densenopersistence} in a constructive way; indeed, the proof explicitly describes how to find the required set.

\begin{proof}[Proof of Proposition~\ref{prop:densenopersistence}]
 By Lemma~\ref{lem:densenopersistence}, $\mathbf{q}_0(o,\{x\})$ only depends on (and it is increasing with respect to) $d(o,x)$ and there exists a sequence 
 $\{r_i\}_{i \in \mathbb{N}}$ of natural numbers such that $\sum_{i \in \mathbb{N}} (1-\mathbf{q}_0(o,\{x_i\})) < +\infty$ for all choices of $\{x_i\}_i$ with $d(o,x_i) \ge r_i$ for all $i \in \mathbb{N}$. Let us enumerate the elements of $\mathbb{T}_d\setminus \{o\}=\{y_1, y_2, \ldots, y_n, \ldots\}$. We choose $\{x_i\}_{i \in \mathbb{N}}$ such that $x_i \in \mathbf{T}_{y_i}$ and $d(o, x_i) \ge r_i$ and we define $A_\emptyset:=\bigcup_{i \in \mathbb{N}}\{x_i\}$. 
 Clearly, $\partial A_\emptyset =\partial \mathbb{T}$ by construction. Moreover,
 $\sum_{i \in \mathbb{N}} (1-\mathbf{q}_0(o,\{x_i\})) < +\infty$, thus, by the Borel-Cantelli's Lemma, only a finite number of vertices of $A_\emptyset$ are visited almost surely by the BRW starting with one particle at $o$.
 
 Observe that, given our choice $\lambda \in (1/d, 1/2\sqrt{d-1}]$ made in Section~\ref{sec:boundarymeasure} (right before Section~\ref{subsec:topologytree}),  survival in a set is equivalent to visiting an infinite number of vertices of the set. This implies that 
 $1=\mathbf{q}(o, A_\emptyset)$ which, in turns, implies
 $1=\mathbf{q}(x, A_\emptyset)$ for all $x \in \mathbb{T}_d$ since the BRW is irreducible.
 Then $\mathbf{q}(A_\emptyset)=\mathbf{q}(\emptyset)$ and the proposition is proven if $B=\emptyset$.

 Consider now a generic $B \subseteq \mathbb{T}_d$. Denote by $\mathcal{S}(A)$ the event ``survival in $A$'' for $A \subseteq \mathbb{T}_d$; the probability of $\mathcal{S}(A)$ for a BRW starting with one particle is $1-\mathbf{q}(x, A)$. Observe that $\mathcal{S}(A \cup B)=\mathcal{S}(A) \cup \mathcal{S}(B)$, therefore
 $\min(\mathbf{q}(x, A),\mathbf{q}(x, B)) \ge \mathbf{q}(x, A \cup B) \ge \mathbf{q}(x, A)+\mathbf{q}(x, B)-1$ for all $x \in X$. 
 If we define $A_B:=A_\emptyset \cup B$ then $\partial A_B =\partial \mathbb{T}_d$ and 
 \[
 \mathbf{q}(B)=
  \min(\mathbf{q}(A_\emptyset),\mathbf{q}(B)) \ge \mathbf{q}(A_B) \ge \mathbf{q}(A_\emptyset)+\mathbf{q}(B)-1=
 \mathbf{q}(B)
 \]
 and the claim is proven.

\end{proof}

Here is another result following from Theorem~\ref{th:1}.
Given $x, y \in \mathbb{T}$, by $\mathbf{T}_{x,y}$ we denote the tree branching from $y$ with respect to the root $x$, that is, $\{z \in \mathbb{T}\colon y \in \varphi_{x,z}\}$. Clearly $y \in \mathbf{T}_{x,y}$.

\begin{Proposition}\label{pro:extensionExtinction}
 Consider a BRW on a tree $\mathbb{T}$. Let $A, B \subseteq \mathbb{T}$ such that $\mathbf{q}(x,A)>\mathbf{q}(x,B)$ for some $x \in \mathbb{T}$. Then
 there exists $y \in \mathbb{T}$ such that
for the BRW starting from $y$ there is positive probability of survival in 
$B$ without ever visiting $\bigcup_{w \in A} \mathbf{T}_{y,w}$.
\end{Proposition}

\begin{proof}

According to Theorem~\ref{th:1} there exists $y \in \mathbb{T}$ such that
there is positive probability of survival in $B$ without visiting $A$. Given the nature of a tree, in order to visit $\bigcup_{w \in A} \mathbf{T}_{y,w}$ the process must visit $A$, since
any path from $y$ to $\mathbf{T}_{y,w}$ visits $w$. Therefore $\mathbf{q}_0(y,A)=\mathbf{q}_0
\Big(y,\bigcup_{w \in A} \mathbf{T}_{y,w} \Big )$ and there is
survival in $B$ without visiting $\bigcup_{w \in A} \mathbf{T}_{y,w}$.
\end{proof}

This applies, for instance, to the set $A=A_\emptyset$ in Proposition~\ref{prop:densenopersistence}; in this case it is not difficult to see that there is
positive probability of
global survival starting from $o$ without ever visiting $\bigcup_{y \in A_\emptyset} \mathbf{T}_{x,y}$.

\subsection{Proofs of Proposition~\ref{pro:connectedequalboundary}, Proposition~\ref{pro:survivalpositivemeasure} and Theorem~\ref{prop:equivalent} (iii)}\label{sec:proofPropconnectedequalboundary}

We start with a useful lemma which  characterizes the boundary of subsets of a tree. We observe that, by using the construction of the Martin boundary of a tree given in Section~\ref{subsec:topologytree}, clearly, $\varphi \in \partial A$ ($A \subseteq \mathbb{T}$) if and only if $A \cap \mathbf{T}_{\varphi (i)} \neq \emptyset$ for all $i \in \mathbb{N}$; in particular, $\varphi  \in \partial \mathbf{T}_x$ if and only if $x=\varphi (i)$ for some $i \in \mathbb{N}$. 
Finally, given $x \in \mathbb{T}, \, \xi \in \partial \mathbb{T}$ we denote by $\varphi _{x,\xi}$ the unique path from $x$ converging to $\xi$.

\begin{Lemma}\label{lem:boundaryA}
 For any subset $B \subseteq \mathcal{M}(P,1)$ the following are equivalent:
 \begin{enumerate}[(1)]
  \item $B$ is a closed subset 
  \item there exists a connected subset  $A \subseteq \mathbb{T}$ such that $B=\partial A$.
 \end{enumerate}

\end{Lemma}

\begin{proof}
We note that $B \subseteq \mathcal{M}(P,1)$ is closed in $\widehat{\mathbb{T}}$ if and only if it is closed in the induced topology on $\mathcal{M}(P,1)$.
%

\noindent $(2) \Longrightarrow (1)$. It follows from the fact that any topological boundary is a closed subset.


\noindent $(1) \Longrightarrow (2)$. Let us define $A:=\{x \in \mathbb{T} \colon \exists \varphi \in 
B,\, i \in \mathbb{N} \text{ such that } \varphi (i)=x\}$. It is clear that $\partial A \supseteq B$. Suppose that $\varphi  \not \in B$, since $B$ is closed then there exists $y \in \mathbb{T}$ such that $\varphi  \in \partial \mathbf{T}_y$ and $\partial \mathbf{T}_y \cap B=\emptyset$. This implies that for every $\varphi^\prime \in B$, $\varphi^\prime \cap \partial \mathbb{T}_y=\emptyset$. Therefore,
by the definition of $A$, $A \cap \mathbf{T}_y =\emptyset$ and then $\varphi \not \in \partial A$.
$A$ is clearly connected, if nonempty, since every vertex is connected to $\varphi(0)=o$.
%
%
\end{proof}

\begin{Remark}\label{rem:connectedborder}
In a tree $\mathbb{T}$ with root $o$ there is another way to construct a set whose boundary is a given closed subset $B$ of $\partial \mathbb{T}$.
Since $B$ is closed, $\partial \mathbb{T} \setminus B=
\bigcup_{x \in I} \partial \mathbf{T}_{x}$
for a suitable choice of $I \subseteq \mathbb{T}_d$ such that $\{\mathbf{T}_{x}\}_{x \in I}$ are mutually disjoint. 
Define $A:=\mathbb{T} \setminus \bigcup_{x \in I} \mathbf{T}_{x}$
It is easy to see that $A$ is connected and $B=\partial A$; indeed,
$\partial A \cap \partial \mathbf{T}_x=\emptyset$ for all $x \in I$, whence $\partial A \subseteq B$. However, if $\varphi \in B$, then $I \cap \varphi=\emptyset$ whence $\varphi \subseteq B$ and then $\varphi \in \partial A$. Clearly, if $B=\emptyset$ then $A=\emptyset$. 

The choice of $\{\mathbf{T}_x\}_{x \in I}$ is not unique and in general $A \supseteq \bigcup_{\varphi \in B} \varphi$, where the latter is the set constructed in Lemma~\ref{lem:boundaryA}. It is easy to prove that there is a unique family $\{\mathbf{T}_x\}_{x \in I}$ such that $\partial \mathbb{T} \setminus B=
\bigcup_{x \in I} \partial \mathbf{T}_{x}$ and where each subtree $\mathbf{T}_x$ is maximal with respect to the set inclusion. With this choice it is not difficult to see that $\mathbb{T} \setminus \bigcup_{x \in I}  \mathbf{T}_{x}=\bigcup_{\varphi \in B} \varphi$.
\end{Remark}

 We describe now the generic connected subset of a tree in the following lemma. Recall that a leaf in a graph is a vertex with only one neighbor.

\begin{Lemma}\label{lem:connected}
Let $A \subseteq \mathbb{T}$ be a connected subset of a tree with root $o$. Denote by $\partial A$ the boundary of $A$, by $L_A$ the set of leaves of $A$
and let $n_A:=\min\{d(o,x) \colon x \in A\}$. Then there exists only one vertex $\bar x \in A$ such that $d(o,\bar x)=n_A$; moreover $A=\{y \in \mathbb{T}\colon y=\varphi _{\bar x, w}(i), \, i \in \mathbb{N}, \, w \in \partial A \cap L_A \}$.
\end{Lemma}
\begin{proof}
Note that if $x,y \in A$ and $A$ is connected then the path $\varphi _{x,y}$ is in $A$.
Denote by $\widetilde A:=\{y \in \mathbb{T}\colon y=\varphi _{\bar x, w}(i), \, i \in \mathbb{N}, \, w \in \partial A \cap L_A \}$.

 If $n_A=0$ then $\bar x=o$. Let $x, y \in A$ such that $d(o,x)=d(o,y) \ge 1$, since $A$ is connected then $x \wedge y \in A$ since it belongs to the path $\varphi _{x,y}$. But $d(o,x \wedge y) \le d(o,x)=n_A$ and by definition of $n_A$  we must have $d(o,x \wedge y) = d(o,x)$ which implies $x=x\wedge y=y$.
 
 By connection every path $\varphi_{\bar x, y}$ where $y \in \partial A \cap L_A$ is in $A$, thus $\widetilde A \subseteq A$. Suppose that $x \in A$; if $x \in L_A$ then $x$ is the ending point of the path $\varphi_{o,x}$ whence $x \in \widetilde A$. Otherwise there is a neighbor $x_1 \in A$ such that $d(o,x_1)=d(o,x)+1$. Suppose we have $x_i \in A$ such that $d(o,x_i)=d(o,x)+i$; either $x_i \in L_A$ (whence $x_i \in \widetilde A$) or there exist a neighbor 
 $x_{i+1} \in A$ such that $d(o,x_{i+1})=d(o,x_{i})+1=d(o,x)+i+1$. If this iterative construction ends in a finite number of steps, say $n$, then $x_{n} \in L_A$ whence $x_n \in \widetilde A$, $x$ belongs to $\varphi(o,x_n)$ thus $x \in \widetilde A$. Otherwise we have and infinite injective path  in $A$ converging to some $y \in \partial A$. Clearly $x$ belongs to $\varphi_{o, y}$ whence $x \in \widetilde A$.
\end{proof}

\begin{proof}[Proof of Proposition~\ref{pro:connectedequalboundary}]

By using Lemma~\ref{lem:boundaryA} or Remark~\ref{rem:connectedborder} we obtain a connected set $B$ such that $\partial B=\partial A$.

We prove now the second part of the proposition.
Let us start by noting that in a connected tree, given two distinct vertices $x \neq y$ a path $\varphi$ connects $x$ to $y$ if and only if $\varphi \supseteq \varphi_{x,y}$, that is, the path $\varphi$ must visit every vertex of the shortest path from $x$ to $y$. This implies that if $A$ is connected and $x, y \in A$ then $\varphi_{x,y} \subseteq A$.

Moreover, given two connected sets $A$ and $B$ such that $A \cap B \neq \emptyset$, clearly $A \cup B$ and $A \cap B$ are connected; in particular if $x, y \in A \cap B$ then $\varphi_{x,y} \subseteq A \cap B$. If we take two disjoint connected components $C_1, C_2 \subseteq A \triangle B$ and a path $\varphi$ going from $C_1$ to $C_2$ then $\varphi \cap A \cap B \neq \emptyset$. Indeed, let $x$ and $y$ be the starting and ending vertices of $\varphi$ 
then $\varphi_{x,y} \subseteq A \cup B$. Moreover, let $i_1$ be such that $\varphi_{x,y}(i_1) \in C_1$ and $\varphi_{x,y}(i) \not \in C_1$ for all $i >i_1$; similarly let $i_2$ be such that $\varphi_{x,y}(i_2) \in C_2$ and $\varphi_{x,y}(i) \not \in C_2$ for all $i <i_2$.
Since $C_1$ and $C_2$ are disjoint connected components and $A \cap B \neq \emptyset$ then $i_2-i_1>1$. Clearly, $x_1:=\varphi_{x,y}(i_1+1) \in A \cap B$ and $y_1:=\varphi_{x,y}(i_2-1) \in A \cap B$;
on the other hand they belong to $A \cup B$. Moreover $\{\varphi_{x,y}(i_1+1), \ldots \varphi_{x,y}(i_2-1)\}$ is the shortest path $\varphi_{x_1,y_1}$
thus, since $A \cap B$ is connected
$\varphi_{x_1,y_1} \subseteq A \cap B$, whence $\{\varphi_{x,y}(i_1+1), \ldots \varphi_{x,y}(i_2-1)\} \subseteq A \cap B$.
Roughly speaking, 
to travel from one connected component of $A \triangle B$ to a different one (when $A \cap B \neq \emptyset$) one needs to visit $A \cap B$.

Let $A$ and $B$ be two connected subset of the tree such that $\partial A=\partial B$.
If $\partial A=\partial B=\emptyset$, that is, $A$ and $B$ are both finite, then $\mathbf{q}(A)=\mathbf{q}(B)$ since the process is irreducible.

Suppose now that $\partial A=\partial B \neq \emptyset$. If $\varphi \in \partial A$ then, since $A$ is connected, $\varphi(i) \in A$ eventually when $i \to +\infty$. A similar argument applies for $B$, thus $\varphi(i) \in A \cap B$ eventually as $i \to +\infty$. In particular $A \cap B \neq \emptyset$.
Note that $A \triangle B$ is the (possibly infinite) union of finite connected components otherwise, an infinite connected component in $A \setminus B$ (resp.~$B \setminus A$) would contain a infinite path, whence a cluster point, which belongs to $\partial A \setminus \partial B$ (resp.~$\partial B \setminus \partial A$).   In this case $
\mathbf{q}(A)=\mathbf{q}(B)$; 
%
%
indeed, since all components are finite, in order to survive in $A \triangle B$ the BRW must visit an infinite number of components. We proved that, in order to move from one component to another, the BRW must visit  $A \cap B$. This means that  survival in $A \triangle B$ implies survival in $A \cap B$. So survival in $A \cup B$ implies survival in $A \cap B$ (the converse is trivial). Whence
$\mathbf{q}(A \cup B)= \mathbf{q}(A \cap B)$; however, $\mathbf{q}(A \cap B) \leq \mathbf{q}(A),\, \mathbf{q}(B) \leq \mathbf{q}(A \cup B)$, thus $\mathbf{q}(A) = \mathbf{q}(B)$.
\end{proof}

%
%

\begin{proof}[Proof of Theorem~\ref{prop:equivalent} (iii)]
  As in Remark~\ref{rem:connectedborder}, consider  $\partial G \setminus \partial U=
\bigcup_{x \in I} \partial \mathbf{T}_{x}$
for a fixed $I \subseteq G$ such that $\{\mathbf{T}_{x}\}_{x \in I}$ are mutually disjoint. Clearly
$\gamma_o(\partial U)=1-\sum_{x \in I} \gamma_o( \partial \mathbf{T}_{x})$; thus
$\gamma_o(\partial U)>0$ if and only if 
$\gamma_o \Big (\bigcup_{x \in I} \partial \mathbf{T}_{x} \Big )<1$.

Define $B:=G \setminus \bigcup_{x \in I} \mathbf{T}_{x}$; we know from Proposition~\ref{pro:connectedequalboundary} that  $B$ is connected, $\partial B=\partial U$ and
$\mathbf{q}(U)=\mathbf{q}(B)$. 
Without loss of generality we can prove the result for $B$.

We know from Lemma~\ref{lemma:rhoVSMmin} that with positive probability there is a line of descent of the BRW converging to $\partial B$. Due to the tree structure and the definition of $B$, then in order to converge to a point of $\partial B$, this line of descendants must cross $B$ infinitely many times, whence $\mathbf{q}(B)<1$.
\end{proof}



\begin{proof}[Proof of Proposition~\ref{pro:survivalpositivemeasure}]
If $A=\emptyset$ or $B=\emptyset$, then Theorem~\ref{prop:equivalent} (iii) yields the result.
Henceforth, we suppose that they are both nonempty.
We know from Lemma~\ref{lem:boundaryA} and  Proposition~\ref{pro:connectedequalboundary} that if we consider the collection of all vertices in the paths from $o$ to the boundary of $A$, namely $A':=\{\varphi(i)\colon \varphi \in \partial A, \, i \in \mathbb{N}\}$, then $\partial A = \partial A^\prime$ and $\mathbf{q}(A^\prime)=\mathbf{q}(A)$. It is enough to prove the result for $A$ and $B$ of this type.
For such $A$ and $B$ clearly $\partial (A \setminus B)=\partial A \setminus \partial B$. Moreover if we consider any connected component of $A \setminus B$, since we are in a tree, is not difficult to see that it can be disconnected from $B$ by removing a single edge.
 Suppose, without loss of generality that $\gamma_o(\partial A \setminus \partial B)>0$ (otherwise, swap roles between $A$ and $B$). Since $A \setminus B$ is an at-most-countable collection of connected components, there is at least one component $C$ such that $\gamma_o(\partial C)>0$. According to Theorem~\ref{prop:equivalent} (iii) there is 
 positive probability of 
 survival in $C$ while the connecting edge is crossed a finite number of times.
 Then, by Theorem~\ref{th:1}, there is 
 positive probability of 
 survival in $C$ without crossing the connected edge, that is, without visiting $B$ (starting from some $x$). Theorem~\ref{th:1} yields $\mathbf{q}(x,A) \le \mathbf{q}(x, C) <
\mathbf{q}(x, B)$ for some $x$.
\end{proof}

\section*{Acknowledgments.}
E.C.\ was partially supported by the project ``Programma per Giovani Ricercatori Rita Levi Montalcini'' awarded by the Italian Ministry of Education.
The authors acknowledge partial support by ``INdAM--GNAMPA Project 2019'', ``INdAM--GNAMPA Project 2020'' and ``INdAM--GNAMPA Project cod.\ 
CUP$\_$ E53C22001930001''. The authors thank the anonymous referee for useful suggestions and comments.

\nocite{}
\bibliographystyle{amsalpha}
\bibliography{bibliography2.bib}

\providecommand{\bysame}{\leavevmode\hbox to3em{\hrulefill}\thinspace}
\providecommand{\MR}{\relax\ifhmode\unskip\space\fi MR }
\providecommand{\MRhref}[2]{%
  \href{http://www.ams.org/mathscinet-getitem?mr=#1}{#2}
}
\providecommand{\href}[2]{#2}
\begin{thebibliography}{{Saw}97}

\bibitem[AD19]{Abraham-Delmas}
Romain Abraham and Jean-Fran\c{c}ois Delmas, \emph{Asymptotic properties of
  expansive {G}alton-{W}atson trees}, Electron. J. Probab. \textbf{24} (2019),
  Paper No. 15, 51. \MR{3916335}

\bibitem[BBHZ22]{BertacchiZuccaAndFriends}
Daniela Bertacchi, Peter Braunsteins, Sophie Hautphenne, and Fabio Zucca,
  \emph{Extinction probabilities in branching processes with countably many
  types: a general framework}, ALEA Lat. Am. J. Probab. Math. Stat. \textbf{19}
  (2022), no.~1, 311--338. \MR{4359794}

\bibitem[BZ09]{bertacchi-zucca-infinite-type}
Daniela Bertacchi and Fabio Zucca, \emph{Characterization of critical values of
  branching random walks on weighted graphs through infinite-type branching
  processes}, J. Stat. Phys. \textbf{134} (2009), no.~1, 53--65. \MR{2489494}

\bibitem[BZ13]{Bertacchi-Zucca-book}
\bysame, \emph{Recent results on branching random walks}, Statistical Mechanics
  and Random Walks: Principles, Processes and Applications, 2013, pp.~289--340.

\bibitem[BZ14]{BZ_stronglocal}
\bysame, \emph{Strong local survival of branching random walks is not
  monotone}, Adv. Appl. Probab. \textbf{46} (2014), no.~4, 400--421.

\bibitem[BZ17]{BZ_genfunapproach}
\bysame, \emph{A generating function approach to branching random walks}, Braz.
  J. Probab. Stat. \textbf{31} (2017), no.~2, 229--253. \MR{MR3635904}

\bibitem[BZ20]{BZ2020}
\bysame, \emph{Branching random walks with uncountably many extinction
  probability vectors}, Braz. J. Probab. Stat. \textbf{34} (2020), no.~2,
  426--438. \MR{4093267}

\bibitem[CH]{CandelleroHutchcroft}
Elisabetta Candellero and Tom Hutchcroft, \emph{{O}n the boundary at infinity
  for branching random walk}, Electronic Communications in Probability, to
  appear.

\bibitem[CR15]{candellero_roberts_BRW}
Elisabetta Candellero and Matthew~I. Roberts, \emph{The number of ends of
  critical branching random walks}, ALEA Lat. Am. J. Probab. Math. Stat.
  \textbf{12} (2015), no.~1, 55--67. \MR{3333735}

\bibitem[GM17]{GilchMueller-EndsBRW}
Lorenz~A. Gilch and Sebastian M\"{u}ller, \emph{Ends of branching random walks
  on planar hyperbolic {C}ayley graphs}, Groups, graphs and random walks,
  London Math. Soc. Lecture Note Ser., vol. 436, Cambridge Univ. Press,
  Cambridge, 2017, pp.~205--214. \MR{3644010}

\bibitem[HLN13]{HautphenneEtAl}
S.~Hautphenne, G.~Latouche, and G.~Nguyen, \emph{Extinction probabilities of
  branching processes with countably infinitely many types}, Adv. in Appl.
  Probab. \textbf{45} (2013), no.~4, 1068--1082. \MR{3161297}

\bibitem[Hut20]{Hutchcroft-NonIntersectionBRW}
Tom Hutchcroft, \emph{Non-intersection of transient branching random walks},
  Probab. Theory Related Fields \textbf{178} (2020), no.~1-2, 1--23.
  \MR{4146533}

\bibitem[Hut22]{Hutchcroft2020Branching}
\bysame, \emph{Transience and recurrence of sets for branching random walk via
  non-standard stochastic orders}, Ann. Inst. Henri Poincar\'{e} Probab. Stat.
  \textbf{58} (2022), no.~2, 1041--1051. \MR{4421617}

\bibitem[KW]{KW2022}
Vadim~A. Kaimanovich and Wolfgang Woess, \emph{Limit distributions of branching
  markov chains}, Ann. Inst. H. Poincar{\'e} Sect. B, to appear.

\bibitem[Loo77]{Lootgieter}
J.~C. Lootgieter, \emph{La {$\sigma $}-alg{\`e}bre asymptotique d'une
  cha{\^{\i}}ne de {G}alton-{W}atson}, Ann. Inst. H. Poincar{\'e} Sect. B
  (N.S.) \textbf{13} (1977), no.~3, 193--230. \MR{0471099}

\bibitem[NS66]{NeySpitzer}
P.~Ney and F.~Spitzer, \emph{The {M}artin boundary for random walk}, Trans.
  Amer. Math. Soc. \textbf{121} (1966), 116--132. \MR{195151}

\bibitem[Ove94]{Overbeck}
Ludger Overbeck, \emph{Martin boundaries of some branching processes}, Ann.
  Inst. H. Poincar\'{e} Probab. Statist. \textbf{30} (1994), no.~2, 181--195.
  \MR{1276996}

\bibitem[PS01]{PemantleStacey}
Robin Pemantle and Alan~M. Stacey, \emph{The branching random walk and contact
  process on galton-watson and nonhomogeneous trees}, Ann. Probab. \textbf{29}
  (2001), no.~4, 1563--1590. \MR{1880232}

\bibitem[PW87]{picardello_woess_trees}
Massimo~A. Picardello and Wolfgang Woess, \emph{Martin boundaries of random
  walks: ends of trees and groups}, Trans. Amer. Math. Soc. \textbf{302}
  (1987), no.~1, 185--205. \MR{887505 (89a:60177)}

\bibitem[PW94]{picardello_woess}
\bysame, \emph{The full {M}artin boundary of the bi-tree}, Ann. Probab.
  \textbf{22} (1994), no.~4, 2203--2222. \MR{1331221 (96f:60133)}

\bibitem[{Saw}97]{sawyer_harmonic_functions}
Stanley~A. {Sawyer}, \emph{{Martin boundaries and random walks.}}, {Harmonic
  functions on trees and buildings. Workshop on harmonic functions on graphs,
  New York, NY, October 30--November 3, 1995}, Providence, RI: American
  Mathematical Society, 1997, pp.~17--44 (English).

\bibitem[Woe00]{woess2000}
Wolfgang Woess, \emph{Random walks on infinite graphs and groups}, Cambridge
  Tracts in Math., vol. 138, Cambridge University Press, Cambridge, 2000.

\bibitem[Woe09]{Woess-denumerable}
\bysame, \emph{Denumerable {M}arkov chains}, EMS Textbooks in Mathematics,
  European Mathematical Society (EMS), Z\"{u}rich, 2009, Generating functions,
  boundary theory, random walks on trees. \MR{2548569}

\bibitem[Zuc11]{Zucca}
Fabio Zucca, \emph{Survival, extinction and approximation of discrete-time
  branching random walks}, J. Stat. Phys. \textbf{142} (2011), no.~4, 726--753.
  \MR{2773785}

\end{thebibliography}

\end{document}